\documentclass[12pt]{article}

\usepackage{amsmath, amssymb, amsthm}
\usepackage{mathrsfs}
\usepackage{hyperref}
\usepackage{tikz}
\usepackage{enumitem}
\usepackage{geometry}
\usepackage{graphicx}
\usepackage{float}
\usepackage{hyphenat}

\theoremstyle{plain}
\newtheorem{theorem}{Theorem}[section]

\newtheorem{proposition}[theorem]{Proposition}
\newtheorem{corollary}[theorem]{Corollary}
\newtheorem{definition}[theorem]{Definition}
\newtheorem{remark}[theorem]{Remark}

\theoremstyle{example}
\newtheorem{example}[theorem]{Example}

\title{On Some Bi-Cayley Graphs over Cyclic Groups of Order $p^2q^2$ and Related Extensions}
\author{Iqbal Adisna Atmaja$^{1}$, Yeni Susanti$^{1}$, Ahmad Erfanian$^{2}$ \\
  {\small $^{1}$Universitas Gadjah Mada, Indonesia} \\
  {\small $^{2}$Ferdowsi University of Mashhad, Iran}}
\date{}

\begin{document}

\maketitle

%\begin{abstract}
%We investigate structural and combinatorial properties of a class of Bi-Cayley graphs defined over cyclic groups of order $p^2q^2$, where $p$ and $q$ are distinct primes. We describe their connectivity, Eulerian property, girth, clique number, chromatic number, and diameter. In particular, we show that the graphs are connected, have girth $3$, and are biregular with well-determined degree. Further observations and open problems are discussed.
%\end{abstract}

\begin{abstract}
    We investigates structural and combinatorial properties of Bi-Cayley graphs defined over cyclic groups of order $p^2q^2$, where $p$ and $q$ are distinct primes. We begin by describing their fundamental group-theoretic underpinnings. The main focus is on analyzing their connectivity, girth, clique number, chromatic number, diameter, and independence number. It is shown that these Bi-Cayley graphs are connected, biregular with explicitly determined degrees, and possess girth three. Furthermore, we prove that their diameter is equal to five. We further extend several results to Bi-Cayley graphs over arbitrary finite groups under suitable restrictions on the connecting set, with particular emphasis on the case where the connecting set consists of all its involutions. These results clarify structural similarities and differences between Cayley graphs and their Bi-Cayley generalizations. 
\end{abstract}

\section{Introduction}
%Graph constructions derived from algebraic structures provide an important bridge between combinatorial methods and algebraic reasoning. Among these, graphs associated with groups, such as Cayley graphs (\cite{cayley}) and their generalizations, are particularly valuable because they encode algebraic information through graph-theoretic properties such as adjacency, connectivity, symmetry, and colorability. The study of these graphs allows structural features of a group to be visualized and analyzed through combinatorial frameworks. 
%===============================================
Graph constructions derived from algebraic structures provide an important bridge between combinatorial methods and algebraic reasoning. A central construction in this area is the Cayley graph, introduced by Cayley in \cite{cayley}, which encodes a group through its connection set and translates algbraic relations into graph adjacency. Many fundamental graph invariants, such as connectivity, girth, clique number, chromatic number, diameter, and independence number, are therefore closely tied to algebraic properties of the underlying group and the choice of the connection set. 

Let $G$ be a finite group with identity element $e_G$, and let $S \subseteq G$ be an inverse-closed subset with $e_G \notin S$. The Cayley graph $\mathrm{Cay}(G,S)$ is the graph with vertex set $G$, where two vertices $x$ and $y$ are adjacent if and only if $xy^{-1} \in S$. Such graphs are $|S|$-regular and vertex transitive, and they are precisely connected when $S$ generates the whole group. When $G$ is cyclic, Cayley graphs admit particularly explicit characterization, allowing many graph parameters to be determined in terms of modular arithmetic, enabling precise analysis.

Bi-Cayley graphs extend this idea by combining two disjoint copies of a group into a single graph with edges determined by three connection sets. Given a finite group $G$ and inverse-closed subsets $S_1, S_2, S_3 \subseteq G$, where $S_1$ and $S_2$ do not contain $e_G$, the Bi-Cayley graph $\Gamma = \mathrm{BiCay}(G; S_1, S_2, S_3)$ has vertex set $\{0,1 \} \times G$. Adjacency is defined so that the induced subgraphs on $\{ 0\} \times G$ and $\{ 1\} \times G$ are the Cayley graphs $\mathrm{Cay}(G, S_1)$ and $\mathrm{Cay}(G, S_2)$, respectively, while the edges between the two parts are determined by $S_3$ such that two vertices $(0,x)$ and $(0,y)$ are adjacent if and only if $xy^{-1} \in S_3$. Thus, a Bi-Cayley graph consists of two Cayley subgraphs whose interaction is governed by an additional connecting set. This construction produces graphs with richer combinatorial behavior than ordinary Cayley graphs.

The present paper focuses on Bi-Cayley graph over cyclic groups of order $p^2q^2$, where $p$ and $q$ are distinct primes. Throughout, we identify that $\mathbb{Z}_{p^2q^2} \cong \mathbb{Z}_{p^2} \times \mathbb{Z_q^2}$, and represent each group element as an ordered pair $(g_p, g_q)$. This decomposition is realized by the mapping $\bar{k} \mapsto (\bar{k} \text{ (mod }p^2), \bar{k} \text{ (mod }q^2))$. This allows adjacency conditions in both Cayley and Bi-Cayley graphs to be analyzed componentwise.

Our choice of connection sets is motivated by earlier work on Cayley graphs defined by elements of prescribed orders. In \cite{tolue15}, Tolue et al. studied Cayley graphs generated by elements of prime order, while in \cite{susanti19}, Susanti and Erfanian analyzed Cayley graphs on cyclic groups of order $p^2q^2$ generated by elements of prime-square order, obtaining explicit results for clique number, chromatic number, and independence number. These results suggest that restricting connection sets by element order yields strong control over graph invariants. In the Bi-Cayley setting, the presence of two Cayley subgraphs introduces additional interactions that significantly influence global properties such as diameter and independence number.

In this paper, we analyze a class of Bi-Cayley graphs over cyclic group $\mathbb{Z}_{p^2q^2}$ whose connection sets are determined by element orders. We determine their connectivity, degree structure, girth, clique number, chromatic number, diameter, and independence number. We write $\omega(G)$, $\chi(G)$, $\mathrm{diam}(G)$ and $\alpha(G)$ for the clique number, chromatic number, diameter, and independence number of the graph $G$, respectively, while we denote the degree of a vertex $v$ as $deg_G(v)$. A key feature of the analysis is the decomposition of the Bi-Cayley graph into two Cayley subgraphs with different structures, whose interplay governs the behavior of the entire graph.

Although the main results are derived in the cyclic setting, several arguments depend only on general group-theoretic properties rather than commutativity. This observation allows parts of the theory to be extended to Bi-Cayley graphs over arbitrary finite groups under suitable restrictions on the connecting sets, clarifying which phenomena are intrinsic to the Bi-Cayley construction and which depend on the arithmetic of cyclic groups.

\subsection{Related Works}
%=================================
Cayley graphs, introduced by Cayley \cite{cayley}, form a central class of graphs in algebraic graph theory due to their close connection with group structure. Their fundamental properties and applications have been extensively studied (see, for instance, \cite{bondy, godsil, wilson}). In particular, many graph invariants of Cayley graphs, such as connectivity and regularity, are determined directly by the generating behavior of the chosen connection set.

Bi-Cayley graphs is a natural generalization of Cayley graphs. Subsequent work by Conder et al. in \cite{conder16} investigated structural properties of Bi-Cayley graphs, including connectivity conditions, automorphism groups, and symmetry-related aspects. These studies established a general framework for Bi-Cayley graphs over arbitrary finite groups, primarily emphasizing structural and transitivity properties. 

When the underlying group is cyclic, the analysis of Cayley and Bi-Cayley graphs becomes more tractable due to the simplicity of subgroup structure and the availability of number-theoretic tools (refer to \cite{dummit, malik, rose78}). This setting has motivated several studies focusing on Cayley graphs defined by elements of prescribed order. Tolue et al. in \cite{tolue15} introduced prime-order Cayley graphs and showed that parameters such as connectivity, diameter, girth, chromatic number, and clique number are strongly influenced by element-order restrictions. Extending this direction, Susanti and Erfanian in \cite{susanti19} studied prime-square-order Cayley graphs on cyclic groups of order $p^2q^2$, obtaining explicit results for independence number, and clique number. Later, Suparwati in \cite{suparwati26} refined the results on cyclic groups of prime-cubic order in a similar approach.

These works demonstrate that imposing order-based restrictions on connection sets yields strong control over the combinatorial structure of Cayley graphs. At the same time, existing results on Bi-Cayley graphs indicate that their global properties are likewise governed by the interaction between the underlying group and the chosen connecting sets \cite{liang08, conder16}. However, precise combinatorial invariants for Bi-Cayley graphs, particularly over cyclic groups and under element order constraints, have received comparatively fewer works.

The present work addresses this gap by studying Bi-Cayley graphs over cyclic groups of order $p^2q^2$ with connection sets defined by element orders. Additionally, we indicate how several results obtained in the cyclic setting extend to Bi-Cayley graphs over more general finite groups, thereby complementing existing literature that primarily focuses on symmetry and connectivity rather than explicit graph parameters.

\sloppy
\section{Results and Discussion}\label{sec:main results}
\subsection{Bi-Cayley Graph Over Cyclic Groups}
%\begin{theorem}\label{cayley connected components}
%    	Let $G$ be a group and let $S \subseteq G$ such that $e_G \notin S$ and $S=S^{-1}$. Set $H := \langle S \rangle.$
%    	Then the components of the Cayley graph $\mathrm{Cay}(G,S)$ are exactly the right cosets $gH$ for every $g \in G$.
%    \end{theorem}
%
%The adjacency relation in a Cayley graph may be interpreted in terms of a group action of $G$ acting on the set $S$. Specifically, for $g \in V(\mathrm{Cay}(G,S))$ and $s \in S$, the vertex $g$ is adjacent to $gs$ which corresponds to the action of the group element $s$ on $g$. This group action induces a single orbit, namely the entire vertex set $V(\mathrm{Cay}(G,S))$. Consequently, the Cayley graph consists of exactly one orbit under the action of $G$.
%
%Cayley graph also posses the vertex-transitive property. That is, for any vertices $u, v \in V(\mathrm{Cay}(G,S))$, there exists a graph automorphism $f$ such that $f(u)=v$. This property of Cayley graph can be further generalized and relaxed, allowing the construction of a related yet structurally richer class of graphs. This consideration motivates the introduction of Bi-Cayley graphs, which preserve certain fundamental features of Cayley graphs while incorporating additional structural complexity. The preceeding subsection is devoted to their formal definition and initial properties.
%
A Bi-Cayley graph may be regarded as a generalization of a Cayley graph obtained by relaxing certain structural conditions. As a consequence, it allows the graph to have two distinct orbit while not necessarily inherit all properties of Cayley graphs. The precise definition of a Bi-Cayley graph is given below. 

\begin{definition}
Let $G$ be a group and let $S_1, S_2, S_3 \subseteq G$ be nonempty such that $S_1$ and $S_2$ are inverse closed and do not contain the identity element of $G$. The \emph{Bi-Cayley graph} $\Gamma = \mathrm{BiCay}(G; S_1, S_2, S_3)$ over a group $G$ is the graph with vertex set $V(\Gamma) = \{0,1\} \times G$ and edge set determined by, 
    \begin{align*}
        (i,x)(j,y) \in E(\Gamma) \iff \begin{cases}
            xy^{-1} \in S_1 \hspace{0.5cm} i = j = 0\\
            xy^{-1} \in S_2 \hspace{0.5cm} i = j = 1\\
            xy^{-1} \in S_1 \hspace{0.5cm} i \neq j, 
        \end{cases}
    \end{align*}
for any two different vertices $(i,x), (j,y) \in V(\Gamma)$.
\end{definition}

The Bi-Cayley graph $\Gamma = \mathrm{BiCay}(G;S_1,S_2,S_3)$ is referred to as Bi-Cayley graph $\Gamma$ over a group $G$. To provide clearer understanding, we give an example of Bi-Cayley graph as follows.

\begin{example}
    Let $\mathrm{Sym}_3$ denote the symmetric group whose elements are denoted as follows, 
    			\begin{align*}
    				\rho_1 = \left(
    				\left.
    				\begin{matrix}
    					1 & 2 & 3 \\
    					1 & 2 & 3
    				\end{matrix} 
    				\right. \right), 
    				\hspace{0.5cm} 
    				\rho_2 &= \left(
    				\left.
    				\begin{matrix}
    					1 & 2 & 3 \\
    					2 & 3 & 1
    				\end{matrix} 
    				\right. \right),
    				\hspace{0.5cm}
    				\rho_3 = \left(
    				\left.
    				\begin{matrix}
    					1 & 2 & 3 \\
    					3 & 1 & 2
    				\end{matrix} 
    				\right. \right), \\
    				\theta_1 = \left(
    				\left.
    				\begin{matrix}
    					1 & 2 & 3 \\
    					1 & 3 & 2
    				\end{matrix} 
    				\right. \right), 
    				\hspace{0.5cm} 
    				\theta_2 &= \left(
    				\left.
    				\begin{matrix}
    					1 & 2 & 3 \\
    					3 & 2 & 1
    				\end{matrix} 
    				\right. \right),
    				\hspace{0.5cm}
    				\theta_3 = \left(
    				\left.
    				\begin{matrix}
    					1 & 2 & 3 \\
    					2 & 1 & 3
    				\end{matrix} 
    				\right. \right).
    			\end{align*}
                Consider the subsets $S_1 = \{\rho_2, \rho_3\}$, $S_2 = \{\theta_1, \theta_2\}$, and $S_3 = \{\rho_1\}$. The corresponding Bi-Cayley graph $\mathrm{BiCay}(\mathrm{Sym}_3; S_1, S_2, S_3)$ is illustrated in the following figure.
			\begin{figure}[H] 
				\centering
				\begin{tikzpicture}[
					every node/.style={circle,
						draw=black,
						thick,
						minimum size=4mm,
						inner sep=0pt,
						font=\scriptsize,
						text=black},
					every path/.style={draw=black, thick}
					]
					\node (a11) at (-2,3) {$(1, \rho_1)$};
					\node (a21) at (-1,2) {$(1, \rho_2)$};
					\node (a31) at (-3,1) {$(1, \rho_3)$};
					\node (b11) at (1,3) {$(1, \theta_1)$};
					\node (b21) at (2,2) {$(1, \theta_2)$};
					\node (b31) at (0,1) {$(1, \theta_3)$};
					\node (a10) at (-2,0) {$(0, \rho_1)$};
					\node (a20) at (-1,-1) {$(0, \rho_2)$};
					\node (a30) at (-3,-2) {$(0, \rho_3)$};
					\node (b10) at (1,0) {$(0, \theta_1)$};
					\node (b20) at (2,-1) {$(0, \theta_2)$};
					\node (b30) at (0,-2) {$(0, \theta_3)$};
					
					\draw (a11) -- (b11);
					\draw (a11) -- (b21);
					\draw (a21) -- (b11);
					\draw (a21) -- (b31);
					\draw (a31) -- (b21);
					\draw (a31) -- (b31);
					\draw (a11) -- (a10);
					\draw (a21) -- (a20);
					\draw (a31) -- (a30);
					\draw (b11) -- (b10);
					\draw (b21) -- (b20);
					\draw (b31) -- (b30);
					\draw (a10) -- (a20);
					\draw (a20) -- (a30);
					\draw (a30) -- (a10);
					\draw (b10) -- (b20);
					\draw (b20) -- (b30);
					\draw (b30) -- (b10);
					
				\end{tikzpicture}
				\caption{Bi-Cayley graph $\mathrm{BiCay}(\mathrm{Sym}_3; S_1, S_2, S_3)$}
				\label{fig:contoh graf bicayley 1}
			\end{figure}
\end{example}

\begin{remark}
    Section ($3$) will specifically consider the \emph{Bi-Cayley} graph $\mathrm{BiCay}(G; S_1, S_2, S_3)$ for a cyclic group $G = \langle a \rangle$ of order $p^2q^2$ for distinct primes $p$ and $q$, with $S_1 = \{ x\in G : |x|= p \text{ or } q\}$, $S_2 = \{ x\in G : |x|= p^2 \text{ or } q^2\}$, and $S_3 = \{ e_G \}$. For brevity, the graph $\mathrm{Cay}(G;S_1, S_2, S_3)$ will simply be denoted as $\Gamma$.    
\end{remark}

%\begin{theorem}\label{elements of order d}
%    Let $G = \langle a \rangle$ be a finite cyclic group of order $n$. The number of elements of order $d$, given $d \mid n$, is $\phi (d)$.  
%\end{theorem}

To provide clearer context, we give the example of Bi-Cayley Graph over cyclic groups as follows.

%Theorem \ref{elements of order d} gives us these results on the size of set $S_1$ and $S_2$. 
%\begin{align*}
%    |S_1| = \phi (p) + \phi(q) = p+q -2 \hspace{0.2cm}\text{ and }\hspace{0.2cm} |S_2| = \phi(p^2) + \phi(q^2) = p(p-1) + q(q-1). 
%\end{align*}

\begin{example}
    Consider the cyclic group $G$ of order $p^2q^2$. Let $p=2$ and $q=3$. We have that $G$ is isomorphic to $\mathbb{Z}_{36}$. We see that $S_1 = \{ \bar{12}, \bar{18}, \bar{24} \}$, $S_2 = \{ \bar{4}, \bar{8}, \bar{9}, \bar{16}, \bar{20}, \bar{27}, \bar{28}, \bar{32}\}$, and $S_3 = \{ \bar{0}\}$. Thus, we have $V(\Gamma) = \{ 0,1\} \times \mathbb{Z}_{36}$ an for each $(0,\bar{x})$, it is adjacent to four other vertices, that is, $(0, \bar{x}+ \bar{12}), (0, \bar{x}+ \bar{18}),(0, \bar{x}+ \bar{24}),$ and $(1,\bar{x})$. Meanwhile, each $(1,\bar{y})$ is adjacent to nine other vertices, that is, $(1, \bar{y} + \bar{4}),(1, \bar{y} + \bar{8}),(1, \bar{y} + \bar{9}),(1, \bar{y} + \bar{16}),(1, \bar{y} + \bar{20}),(1, \bar{y} + \bar{27}),(1, \bar{y} + \bar{28}),(1, \bar{y} + \bar{32}),$ and $(0,\bar{y})$. Therefore, the induced subgraphs may be illustrated as follows. 
    \begin{figure}[H]
        \centering
        \begin{minipage}{0.45\textwidth}
            \centering
            \includegraphics[width=\linewidth]{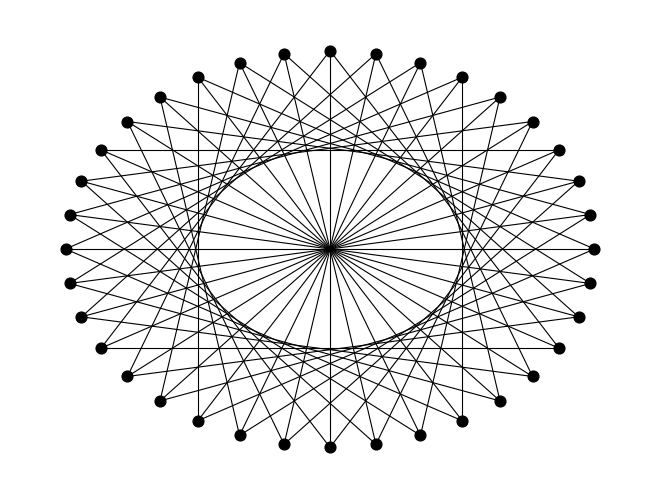}
            \caption{Subgraph induced by $\{ 0\} \times \mathbb{Z}_{36}$}
            \label{fig:gambar1}
        \end{minipage}%
    \hfill
        \begin{minipage}{0.45\textwidth}
            \centering
            \includegraphics[width=\linewidth]{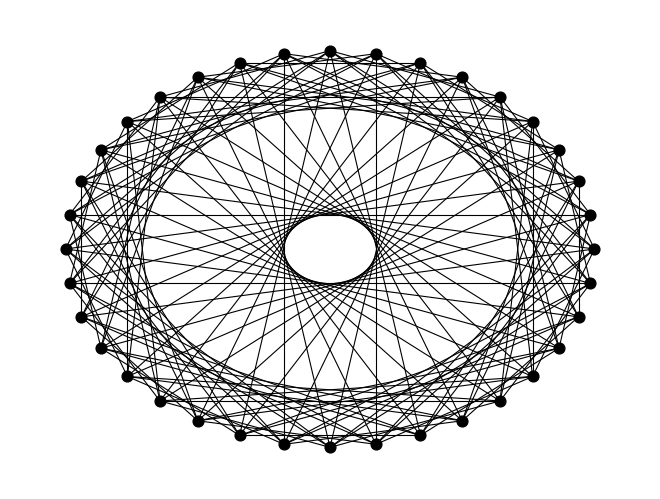}
            \caption{Subgraph induced by $\{ 1\} \times \mathbb{Z}_{36}$}
            \label{fig:gambar2}
        \end{minipage}
    \end{figure}

    The Bi-Cayley graph $\Gamma$ consists of two subgraph presented in Figure \ref{fig:gambar1} and Figure \ref{fig:gambar2}. As for every $(0,\bar{x})$ in the subgraph induced by $\{0 \} \times G$ is adjacent to $(1, \bar{x})$ in the subgraph induced by  $\{1 \} \times G$, the whole Bi-Cayley graph $\Gamma$ is as illustrated. 
    \begin{figure}[H]
        \centering
        \includegraphics[width=0.8\linewidth]{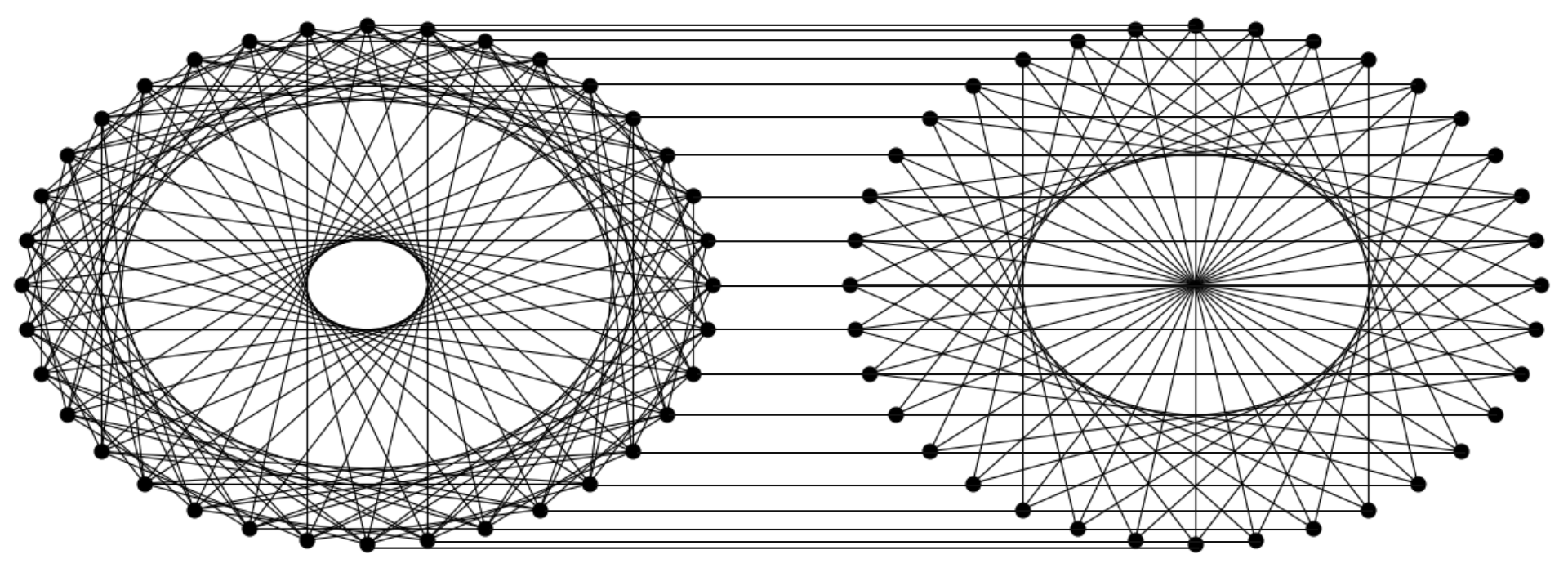}
        \caption{Bi-Cayley over $G \cong \mathbb{Z}_{36}$}
        \label{fig:placeholder}
    \end{figure}
\end{example}

%(Cayley graph connectivity criterion, Euler’s totient, etc.)
%
%This section discusses the properties of Bi-Cayley graphs $\Gamma$, including the general structure, graph connectivity, Eulerian properties, girth number, clique number, chromatic number, graph diameter, and independence number. 
%
Up to this point, we have defined the Bi-Cayley graph and described how adjacency is determined through $S_1,S_2$, and $S_3$. A natural question arises: how does this structure relate to the Cayley graph itself? Interestingly, a Bi-Cayley graph can be viewed as being constructed from two Cayley graphs on $G$. This explains the use of the prefix “Bi”, as the Bi-Cayley graph consists of two such copies connected according to $S_3$.

\begin{proposition}
The Bi-Cayley graph $\Gamma = \mathrm{BiCay}(G; S_1,S_2,S_3)$ consists of two Cayley subgraphs.
\end{proposition}

\begin{proof}
The vertex set of $\Gamma$ decomposes into two disjoint sets $\{0\} \times G$ and $\{1\} \times G$. It is clear that the subgraph induced by $\{0\} \times G$ is the Cayley graph $\mathrm{Cay}(G,S_1)$ and the subgraph induced by $\{ 1\}\times G$ is the Cayley graph $\mathrm{Cay}(G,S_2)$. Therefore the Bi-Cayley graph $\Gamma = \mathrm{BiCay}(G;S_1,S_2,S_3)$ consists of two Cayley subgraphs. 
\end{proof}

\begin{remark}
		Since $\{0\} \times G$ and $\{1\} \times G$ are a group and are also isomorphic to $G$, we essentially have that $V(\Gamma_0) = V(\Gamma_1) = G$. Hence we can see the set  $S_1$ and $S_2$ as a subset of $\{0\} \times G$ and $\{1\} \times G$ (up to isomorphism). For subsequent discussion, the Cayley subgraph $\Gamma_0$ will be referred to as the $0$-\textit{side subgraph}, and the Cayley subgraph $\Gamma_1$ will be referred to as the $1$-\textit{side subgraph}. These notation may be used interchangeably with $\Gamma_0$ and $\Gamma_1$, respectively, since they represent the same structures.
	\end{remark}

After describing the general structure of the Bi-Cayley graph $\Gamma$, we now proceed to perform basic countings on $\Gamma$. These countings include determining the number of vertices, the number of edges, and the degree of each vertex in the graph. All of these results are summarized in the following proposition.

\begin{proposition}
    The Bi-Cayley graph $\Gamma$ has $2p^2q^2$ vertices, is biregular of degree $p(p-1)+ q(q-1)+1$ and $p+q -1$ , and has $\dfrac{p^2q^2(p^2 + q^2)}{2}$ edges.    
\end{proposition}
 
%\begin{proof}
%Note that the vertices of the Bi-Cayley graph is the set $\{0,1 \} \times G$ hence the number of vertices is $|\{0,1 \} \times G|$ which is equal to $2 \cdot p^2q^2$. The $0$-side is a Cayley graph so it is $|S_1|$-regular, as is the $1$-side is $|S_2|$-regular. Since there is a perfect matching in each vertex between the $0$ and $1$ side, each vertex on the $0$-side has a degree of $|S_1|+1$ and each vertex on the $1$-side has a degree of $|S_2|+1$. Thus, the full Bi-Cayley graph is biregular graph with degree $|S_1|+1 = p+q-1$ and $|S_2|+1 = p(p-1) + q(q-1) + 1$. \\
%From the \emph{Handshaking Lemma}, we have that 
%\begin{align*}
%    |E(\Gamma)| = \dfrac{1}{2} \sum_{v \in V(\Gamma)} deg(v) = \dfrac{p^2q^2(p^2 + q^2)}{2},
%\end{align*}
%which completes the proof.
%\end{proof}

We examine the connectivity of the Bi-Cayley graph $\Gamma$. As we later discuss, rather than analyzing the entire graph at once, we can determine its connectivity by focusing on one of its induced Cayley subgraphs, namely $\Gamma_1$.

\begin{proposition}
    The $\Gamma_1$ subgraph of the Bi-Cayley graph $\Gamma$ is connected.
\end{proposition}

\begin{proof}
    Notice that the set $S_2$ contains element of order $p^2$ and $q^2$, that is, $a^{q^2}$ and $a^{p^2}$, respectively. Observe that, 
            $\langle a^{p^2}, a^{q^2} \rangle = \langle a^{\gcd (p^2,q^2)} \rangle = \langle a^1 \rangle = G.$
    Thus, we see that $G = \langle a^{p^2}, a^{q^2} \rangle \subseteq \langle S_2 \rangle$. This means that the set $S_2$ generates group $G$. So, the $\Gamma_1$ subgraph is connected.
\end{proof}

Utilizing the connectivity property of the subgraph $\Gamma_1$, we are able to infer the connectivity of the entire Bi-Cayley graph $\Gamma$. This is formally stated in the following theorem.

\begin{theorem}
The Bi-Cayley graph $\Gamma$ is connected.
\end{theorem}

\begin{proof}
    Let $(i,x), (j,y) \in V(\Gamma)$.   
    \begin{enumerate}
        \item [$(i)$] if $i = j = 1$, then $(i,x)$ and $(j,y)$ lies on the connected subgraph $\Gamma_1$, hence there exists a path connecting them. 
        \item [$(ii)$] if $i = j = 0$, then there exists path $(i, x) - (1,x) - \dots - (1,y) - (j,y)$ connecting the two. The path from $(1,x)$ to $(1,y)$ exists because the $\Gamma_1$ subgraph is connected. 
        \item [$(iii)$] if $i \neq j$, without the loss of generality, let $i = 0$, then there exists path $(i,x) - (j,x) - \dots - (j,y)$ connecting the two. 
    \end{enumerate}
    From all possible cases, we can conclude that the Bi-Cayley graph $\Gamma$ is connected. 
\end{proof}

Since the Bi-Cayley graph $\Gamma$ is connected, we can determine its Eulerian property. See proposition below. 

\begin{proposition}
The Bi-Cayley graph $\Gamma$ is not Eulerian.
\end{proposition}

\begin{proof}
Since each degree of the vertices in $\Gamma_1$ is $p(p-1)+q(q-1)+1$ and it yields odd value for any distinct prime $p,q$, then by Dirac's Theorem \cite{wilson}, we can conclude that the graph $\Gamma$ is not Eulerian. 
\end{proof}

The girth provides insight into the cyclic structure of the graph, specifically identifying the length of the shortest cycle contained within it. Before deriving the girth of the Bi-Cayley graph, we first present the following proposition, which will serve as a key component in the subsequent proof.  

\begin{proposition}
    Every cycle of length $3$ of the graph $\Gamma$ must lie entirely on the same side, either $\Gamma_0$ or $\Gamma_1$. 
\end{proposition}

\begin{proof}
    Without the loss of generality, suppose the cycle starts on the $0$-side. Suppose that the cycle starts from vertex $(0,g) \in V(\Gamma)$, if the cycle goes through the $1$-side then the minimum step that the path goes back to $(0,g)$ is of length $4$. The path with minimum possible length is $(0,g) \to (1,g) \to (1,h) \to (0,h) \to (0,g)$, provided $(1,g)$ is adjacent to $(1,h)$ and $(0,g)$ is adjacent to $(0,h)$ for some $h \in G$ and $h \neq g$. Thus, there could not be any cycle of length $3$ jumping from one side to another. 
\end{proof}

We claim that there is a cycle of length $3$ and so the girth of the graph is indeed $3$. This cycle must lie on one side. We then conclude that the girth number of the entire Bi-Cayley structure is the minimum of the girth number between the two sides.

We may proceed to prove the claim above. Formally, it is stated in this following theorem.

\begin{theorem}
The girth number of $\Gamma$ is $3$.
\end{theorem}

\begin{proof}
We are given $|G|=p^2q^2$, without the loss of generality, suppose $p < q$ implying $q \geq 3$. Because $q$ is prime number that divides $|G|$, then there exists a unique subgroup $H$ of $G$ such that the order of $H$ is equal to $q$. It follows that $H$ is also a cyclic group and contains $q-1$ nonidentity element. These nonidentity element is of order $q$. Thus, there exists at least two different element $s,t \in H \setminus \{ e_G\}$ such that for every $g \in G$ there exists cycle of length $3$, i.e., 
    $(0, g) \to (0,gs) \to (0, gt) \to (0,g).$
Thus, the girth of $\Gamma$ is $3$. 
\end{proof}
%
%In this subsection, we investigate the cliques of the Bi-Cayley graph $\Gamma$, i.e., the complete subgraphs induced by the subset of the vertex set. Moreover, the size of the largest clique, known as the clique number, provides information related to other graph parameters such as the chromatic number. We proceed with the following remark.
%
\begin{remark}
		It has been observed that the edges connecting the subgraph $\Gamma_0$ and the subgraph $\Gamma_1$ are precisely the edges between vertices of the form $(0,g)$ and $(1,g)$ for each $g \in G$. Consequently, no vertex in $\Gamma_0$ is adjacent to more than one vertex in $\Gamma_1$, and vice versa. This implies that any clique of size greater than $2$, if any, must be contained entirely within a single part of the graph, either $\Gamma_0$ or $\Gamma_1$.
\end{remark}

Thus, it suffices to consider the clique number of each $\Gamma_0$ and $\Gamma_1$ to determine the clique number of the entire graph. Another observation to be made is that every vertex in the subgraph $\Gamma_1$ has the form $(1,g)$ for some $g \in G$. Within $\Gamma_1$, the first coordinate, i.e., the number $1$, plays no role in determining adjacency. Therefore, it is convenient to suppress this coordinate and denote each vertex $(1,g)$ simply by $g$. This analoguosly also works for the subgraph $\Gamma_0$. Thus, in the subgraph $\Gamma_0$, the vertex $(0,g)$ is simply written as $g$.

\begin{theorem}
The clique number of $\Gamma$ is $\max \{p,q \}$.
\end{theorem}

\begin{proof}
We will determine the clique on each subgraphs. We see that on the $\Gamma_0$, any two vertices $g, h \in V(\Gamma_0)$ is adjacent if and only if the order $gh^{-1}$ is either $p$ or $q$. Now, we see the vertices as pair of element in $\mathbb{Z}_{p^2} \times \mathbb{Z}_{q^2}$. Notice that if $g$ and $h$ have a different $p$-component and $q$-component, the order of $gh^{-1}$ divides both $p$ and $q$, therefore the order of $gh^{-1}$ does not equal to $p$ or $q$. This means that the clique entirely lies on either $Q_\alpha := \{ (x_p, \alpha): x_p \in \mathbb{Z}_{p^2}\}$ or $P_\beta := \{ (\beta, x_q): x_q \in \mathbb{Z}_{q^2}\}.$
    %\begin{align*}
    %    Q_\alpha := \{ (x_p, \alpha): x_p \in \mathbb{Z}_{p^2}\} \text{ or } P_\beta := \{ (\beta, x_q): x_q \in \mathbb{Z}_{q^2}\}. 
    %\end{align*}
    
Suppose that the clique $T$ lies entirely on $Q_{\alpha}$. We can see $Q_\alpha$ as a subset of $\mathbb{Z}_{p^2}$. Let $\bar{k}^i, \bar{k}^j \in \mathbb{Z}_{p^2}$. It can be showed that $|\bar{k}^i(\bar{k}^j)^{-1}| = p$ occurs if and only if $p \mid i-j \text{ but } p^2 \nmid i - j.$
    %\begin{align*}
     %   |\bar{k}^i(\bar{k}^j)^{-1}| = p \iff p \mid i-j \text{ but } p^2 \nmid i - j. 
    %\end{align*}
Therefore, the biggest such clique set can be is, $T = \{ \bar{k}^{r}, \bar{k}^{r+p}, \bar{k}^{r+2p}, \dots , \bar{k}^{r+(p-1)p} \} \subseteq \mathbb{Z}_{p^2}.$ 
    %\begin{align*}
    %    T = \{ \bar{k}^{r}, \bar{k}^{r+p}, \bar{k}^{r+2p}, \dots , \bar{k}^{r+(p-1)p} \} \subseteq \mathbb{Z}_{p^2}.
    %\end{align*}
Hence, the clique number is $p$. On the case that the clique lies on the $P_\beta$ yields clique number $q$. Thus, the clique number of the $\Gamma_0$ is $\max\{p,q\}$.

Going by the same logic, we suppose that the clique $R$ on the $\Gamma_1$ lies entirely on $Q_\alpha$. We see that for any two vertices $\bar{k}^i, \bar{k}^j \in \mathbb{Z}_{p^2}$, $|\bar{k}^i (\bar{k}^j)^{-1}| = p^2 \iff p \nmid i-j.$
    %\begin{align*}
    %    |\bar{k}^i (\bar{k}^j)^{-1}| = p^2 \iff p \nmid i-j.
    %\end{align*}
So, the biggest such clique set is $R = \{ \bar{k}^0, \bar{k}, \bar{k}^2, \dots, \bar{k}^{p-1} \} \subseteq \mathbb{Z}_{p^2}.$
    %\begin{align*}
    %    R = \{ \bar{k}^0, \bar{k}, \bar{k}^2, \dots, \bar{k}^{p-1} \} \subseteq \mathbb{Z}_{p^2}.
    %\end{align*}
In the case that the clique of the $\Gamma_1$ lies on $P_\beta$, we have clique of size $q$. Thus, the clique number on the $\Gamma_1$ is the same as the $\Gamma_0$, that is, $\max \{p,q \} $.

Since $p,q \geq 2$ and no clique greater than $2$ lies on different side, we may conclude that the clique number of $\Gamma$ is indeed $\max \{p,q \}$.
\end{proof}

This result on the clique number also agrees with the former result on the girth number. For any primes $p$ and $q$, we have that $\omega(\Gamma) = \max \{p,q \} \geq 3$. So, the smallest clique is of size $3$, which is a cycle of length $3$. Thus, the girth number is also $3$, which we already proved.

It is already established that the lower bound of the chromatic number is $\omega(\Gamma) = \max \{p,q\}$. We now show that the chromatic number of $\Gamma$ is $\max \{p,q \}+1$.

We define two subset of $G$ that will be useful in determining the chromatic number of the $1$-side. Define $S_A = \{x \in G : |x|=p^2 \}$ and $S_B = \{ x \in G : |x| = q^2\}.$
%\begin{align*}
%    S_A = \{x \in G : |x|=p^2 \} \text{ and } S_B = \{ x \in G : |x| = q^2\}.
%\end{align*}
It is clear that the set $S_A$ and $S_B$ are inverse closed and do not contain the identity element of $G$. We will proceed to the proposition below.
\begin{proposition}
    The subgraph $\Gamma_1$ of $\Gamma$ is a cartesian product of two Cayley graph $A:= \mathrm{Cay}(\mathbb{Z}_{p^2}, S_A)$ and $B:= \mathrm{Cay}(\mathbb{Z}_{q^2}, S_B)$. 
\end{proposition}

\begin{proof}
    Let $g = (g_p,g_q)$, $h =(h_p, h_q) \in G$. The vertices $g$ and $h$ are adjacent on the $\Gamma_1$ if and only if the order of $gh^{-1}$ is either $p^2$ or $q^2$. In other words, $g$ and $h$ will be adjacent to each other if either of these cases happens,  
    \begin{enumerate}
        \item [$(i)$] $|g_ph_p^{-1}| = p^2$ and $g_q = h_q$, or
        \item [$(ii)$] $|g_qh_q^{-1}| = q^2$ and $g_p = h_p$.
    \end{enumerate}
    This adjacency coincides with the adjacency of cartesian product of graph $A$ and graph $B$. Thus, the subgraph $\Gamma_1$ is exactly the cartesian product of $A$ and $B$. We may denote it by $A \square B$. 
\end{proof}

%\begin{theorem}
%    Let $G_1$ and $G_2$ be a graph. Then 
%    \begin{align*}
%        \chi(G_1 \square G_2) = \max \{\chi (G_1), \chi (G_2) \}.
%    \end{align*}
%\end{theorem}

It suffices to find the chromatic number of $A$ and $B$ to determine the chromatic number of the $\Gamma_1$. To do so, we first investigate the structure of the Cayley graph $A$ and $B$. 

\begin{proposition}
    The Cayley graph $A$ is isomorphic to complete $p$-partite graph $K_{p,p,\dots , p}$ and the Cayley graph $B$ is isomorphic to complete $q$-partite graph $K_{q,q, \dots , q}$. 
\end{proposition}

\begin{proof}
    We will prove the case of Cayley graph $A = \mathrm{Cay}(\mathbb{Z}_{p^2}, S_A)$. Two vertices $\bar{k}^i, \bar{k}^j \in V(A)$ is adjacent if and only if $|\bar{k}^{i-j}| = p^2$ or equivalently $i \not\equiv j \text{ (mod }p)$. %The vertices of $A$ can be written as an entry of a matrix, 
        %\begin{align*}
        %   \begin{bmatrix}
        %       \bar{k}^{0} & \bar{k}^{p} & \bar{k}^{2p} &\cdots & \bar{k}^{(p-1)p} \\
        %       \bar{k}^{1} & \bar{k}^{p+1} & \bar{k}^{2p+1} &\cdots & \bar{k}^{(p-1)p+1} \\
        %       \vdots & \vdots & & \ddots & \vdots \\
        %       \bar{k}^{p-1} & \bar{k}^{2p-1} & \cdots &\cdots & \bar{k}^{(p-1)p} \\
        %   \end{bmatrix} ,
        %\end{align*}
    %in which each row represents the same residue class modulo $p$. Each vertices on a same residue class are not adjacent. 

    Let $A_i = \{ \bar{k}^n : n \equiv i \text{ (mod }p) \}$. Then, $V(A) = \displaystyle\bigcup_{i=0}^{p-1} A_i$ where $A_i \cap A_j = \emptyset$ for $i \neq j$. Let $\bar{k}^n, \bar{k}^{n'} \in A_i$, we see that $\bar{k}^n \bar{k}^{n'} \notin E(A)$. On the other hand, let $\bar{k}^{m_i} \in A_i$ and $\bar{k}^{m_j} \in A_j$ where $i \neq j$, it follows that $\bar{k}^{m_i}\bar{k}^{m_j} \in E(A)$. This is precisely the definition of complete $p$-partite graph denoted by $K_{p,p,\dots , p}$. Analoguously, we also obtain that $B$ is a complete $q$-partite graph $K_{q,q,\dots , q}$. 
\end{proof}

\begin{proposition}
    The chromatic number of $\Gamma_1$ is $\mathrm{max}\{p,q \}$.
\end{proposition}

%\begin{proof}
%    It is clear that $\chi (K_{p,p,\dots , p})= p$ and $\chi (K_{q,q,\dots , q})= q$. Thus, the chromatic number of the $\Gamma_1$ is the maximum of the chromatic number of $A$ and $B$ which is equal to $\max \{ p,q \}$.
%\end{proof}
  
\vspace{0.1cm}

Now, let us define the proper coloring of $\Gamma_1$. First of all, we define proper coloring on the Cayley graphs $A$ and $B$. We define $c_A : V(A) \rightarrow \{0, 1, 2, \dots , k-1 \}$ and $c_B : V(B) \rightarrow \{0, 1, 2, \dots , k-1 \}$ with $k = \mathrm{max}\{p,q \}$ and 
    \begin{align*}
        c_A(\bar{k}^i) \equiv i \text{ (mod }k), \hspace{0.2cm } \text{ and} \hspace{0.2cm} c_B(\bar{k}^j) \equiv j \text{ (mod }k)
    \end{align*}
for every $\bar{k}^i \in V(A)$ and $\bar{k}^j \in V(B)$. This is a proper coloring of $A$ and $B$ because each residue class are assigned different colors. Next, we can define coloring $col : V(A \square B) \rightarrow \{ 0,1,2, \dots , k-1 \}$ as 
    \begin{align*}
        col((\bar{k}^i , \bar{k}^j)) \equiv (c_A(\bar{k^i}) + c_B(\bar{k}^j)) \text{ (mod } k)
    \end{align*}
for every $(\bar{k}^i, \bar{k}^j) \in V(A \square B)$. This is a proper coloring of $\Gamma_1$ that only utilizes $k = \mathrm{max} \{p,q \}$ colors. 

\begin{proposition}
    The chromatic number of $\Gamma_0$ is $\mathrm{max}\{p,q \}$.
\end{proposition}

\begin{proof}
    Without the loss of generality, suppose that $p>q$. Define the coloring $c_0 : V(\Gamma_0) \rightarrow \{0, 1, \dots , p-1\}$ by 
        \begin{align*}
            c_0(g) \equiv \left\lfloor \dfrac{g}{pq^2} \right\rfloor \text{ (mod }p)
        \end{align*}
    for all $g \in V(\Gamma_0)$. It can be shown that $c_0$ is a proper $p$-coloring of $\Gamma_0$.
\end{proof}

\begin{definition}
    Let $G$ be a cyclic group of order $p^2q^2$. We use the following notations regarding the subgroup of $G$. Define $K_p$, $K_q$, $H_p$, and $H_q$ to be the unique subgroups of order $p$, $q$, $p^2$, and $q^2$, respectively.
\end{definition}

For the remainder of the discussion concerning the chromatic number of the Bi-Cayley graph over cyclic groups, we assume that $p > q$, hence $\max \{p,q\}=p$. The argument is symmetric in the case that $p < q$. 

\begin{proposition}
    Each left coset of $K_p$ has size $p$ and forms a clique in $\Gamma_0$
\end{proposition}

\begin{proof}
    Fix a coset $C = gK_p = \{gh : h \in K_p \}$. If $gh_1,gh_2 \in C$ with $h_1 \neq h_2$, then $gh_1(gh_2)^{-1} = h_1h_2^{-1} \in K_p \setminus \{e_G \}$. Thus $|gh_1(gh_2)^{-1}|=p$. By the adjacency rule in $\Gamma_0$, $(0,gh_1)$ is adjacent to $(0,gh_2)$. Therefore, all pairs in $C$ are adjacent, meaning that $C$ is a clique of size $p$.
\end{proof}

There are $pq^2$ distinct $K_p$ cosets, each of size $p$. Next, we present the characteristics on how colors appear on $K_p$ cosets in any proper $p$-coloring of $\Gamma_0$

\begin{proposition}
    If $c_0$ is a proper $p$-coloring of $\Gamma_0$, then on each $K_p$ coset, the $p$ vertices receive all $p$ distinct colors meaning each color occurs exactly once in each coset.
\end{proposition}

\begin{theorem}
    The Bi-Cayley graph $\Gamma$ is not $\max\{p,q\}$-colorable.
\end{theorem}

\begin{proof}
    Let $\max\{p,q\}=p$. For the sake of contradiction, assume that there exists proper $p$-colorings $c_0, c_1$ of $\Gamma_0$ and $\Gamma_1$, respectively, such that $c_0(g) \neq c_1(g)$ for all $g \in G$. This condition is required as $(0,g)$ is adjacent to $(1,g)$ for all $g \in G$. 

    Fix a color $t \in \{0, 1, \dots, p-1\}$ for $c_0$. There are exactly $pq^2$ elements $g$ with $c_0(g) = t$. It follows that none of those elements can have $c_1(g)=t$. So, these $pq^2$ elements are distributed among the remaining $p-1$ colors under $c_1$. 

    Let $A_t = \{ g \in G : c_0(g) = t \}$. It is already established that $|A_t|=pq^2$. For each $s \in \{0,1,\dots, p-1 \}$, define the non-negative integer $n_{t,s} = |\{ g \in G : c_0(g)=t \text{ and } c_1(g) = s\}|.$
        %\begin{align*}
        %    n_{t,s} = |\{ g \in G : c_0(g)=t \text{ and } c_1(g) = s\}|.
        %\end{align*}
    Because every $g \in A_t$ has some $c_1(g)$, we have
        \begin{align*}
            \sum_{s=0}^{p-1}n_{t,s} = |A_t|=pq^2. 
        \end{align*}
    But, by the constraint $c_1(g) \neq c_0(g)$, we have $n_{t,t}=0$. So,
        \begin{align}
            \sum_{s=0, s\neq t}^{p-1}n_{t,s} = pq^2. \tag{*}
        \end{align}
    There are $p-1$ terms on the left hand side of $(*)$, and thus the average of these $p-1$ integers are $\dfrac{pq^2}{p-1}$. Therefore, at least one of the $n_{t,s}$ with $s\neq t$ is greater than the average. This implies that there exists $s \neq t$ such that 
        \begin{align*}
            n_{t,s} \geq \left\lceil \dfrac{pq^2}{p-1}. \right\rceil \geq q^2 + 1.
        \end{align*}
    Thus, for each $t$, there exists some $s \neq t$ with at least $q^2 + 1 $ elements satisfying $c_0(g) = t$ and $c_1(g) = s$. We call such pair $(t,s)$ a \textit{coloring pair} for $t$. 

    Now, partition $G$ into $q^2$ distinct cosets of $H_p$ in which each coset has size $q^2$. Take any coloring pair $(t,s)$. There are at least $q^2 + 1$ elements with $(c_0(g), c_1(g))=(t,s)$. These $q^2 + 1$ elements are placed into $q^2$ different $H_p$ cosets. By the \textit{Pigeonhole Principle}, there must be an $H_p$ coset that contains at least two of them. Therefore, there exist distinct elements $x, y $ lying in the same $H_p$ coset such that $c_0(x)=c_0(y) = t$ and $c_1(x)=c_1(y) = s$. 

    Since $x$ and $y$ lie in the same $H_p$ coset and $c_0$ is a proper coloring, we have that $|xy^{-1}|=p^2$. This implies that $(1,x)$ is adjacent to $(1,y)$. But we already have that $c_1(x) = c_1(y) = s$. That contradicts the assumption that $c_1$ is a proper coloring of $\Gamma_1$. Therefore, no such pair of colorings $c_0, c_1$ can exist. Hence the whole Bi-Cayley graph is not $p$-colorable, i.e., not $\max\{p,q \}$-colorable.
\end{proof}

\begin{theorem}
    The chromatic number of Bi-Cayley graph $\Gamma$ is $\max\{p,q\}+1$.
\end{theorem}

\begin{proof}
    Let $k = \max\{ p,q\}$. Define the coloring $c_{\Gamma}: V(\Gamma) \rightarrow \{0,1, \dots, k\}$ as follows, 
        \begin{align*}
            c_{\Gamma}((i,g)) = \begin{cases}
                col(g) \text{ (mod }k), & i =1\\
                c_0^*(g), & i =0
            \end{cases}
        \end{align*}
    for all $(i,g) \in V(\Gamma)$ where $c_0^* : V(\Gamma_0) \rightarrow \{0,1,\dots, k\}$ defined by
        \begin{align*}
            c_0^*(g) = \begin{cases}
                c_0(g) \text{ (mod }k), &c_0(g) \neq col(g)\\
                k, & c_0(g) = col(g).
            \end{cases}
        \end{align*}
    This coloring trivially ensures that all adjacent vertices in $\Gamma$ are assigned different colors.
\end{proof}

The diameter is a fundamental graph invariant that measures the maximum distance between any two vertices within the graph. Understanding the diameter provides insight into the overall structure of $\Gamma$, particularly regarding how efficiently information or connectivity can be propagated across the graph. We first present several results regarding the structural property of each subgraph.

\begin{proposition}
    The subgraph $\Gamma_0$ is a disconnected graph consisting of $pq$ disjoint components.
\end{proposition}

In order to simplify the notation, we may write $\equiv_n$ to denote equivalence relation modulo $n$. 
Moreover, we have that, for each $(g_p, g_q) \in V(\Gamma_0)$, the subgraphs induced by the sets of the form
    \begin{align*}
        \mathbf{C}_{(g_p, g_q)} = \{ (x', y') \in G : g_p \equiv_p x', g_p \not\equiv_{p^2} x' \text{ and } g_q \equiv_q y', g_q \not\equiv_{q^2} y' \}
    \end{align*}
are the disjoint components of $\Gamma_0$. In subsequent discussion, said component is denoted by $G_{(g_p, g_q)}$. Further characterization of each component is provided in the following. 

\begin{proposition}
    For any $g = (g_p, g_q) \in G$, the induced subgraph $G_{(g_p, g_q)}$ is isomorphic to the cartesian product graph $K_p \square K_q$.
\end{proposition}

Let $V(K_p \square K_q) = \{ (u_i, v_j) : i = 0,1, \dots , p-1; j = 0,1, \dots , q-1 \}$, the graph isomorphism is realized by the bijective mapping $\zeta : V(K_p \square K_q) \to V(G_{(g_p, g_q)})$ defined by 
    \begin{align*}
        \zeta((u_i, v_j)) \equiv (g_p + ip \text{ (mod }p^2), g_q + jq \text{ (mod }q^2) )
    \end{align*}
for all $(u_i, v_j) \in V(K_p \square K_q)$.

Combining the fact that $\mathrm{diam}(G_1 \square G_2) = \mathrm{diam} (G_1) + \mathrm{diam} (G_2)$, provided that $G_1$ and $G_2$ are connected, we have the following assertions. 

\begin{proposition}
    Let $(g_p, g_q) \in G$ and let $(x,y)$, $(x',y')$ be distinct vertices in $G_{(g_p, g_q)}$. Then, $d_{G_{(g_p, g_q)}}((x,y),(x',y')) = \mathbf{1}_{x \neq x'} + \mathbf{1}_{y \neq y'},$
    %\begin{align*}
    %    d_{G_{(g_p, g_q)}}((x,y),(x',y')) = \mathbf{1}_{x \neq x'} + \mathbf{1}_{y \neq y'},
    %\end{align*}
    where $\mathbf{1}$ is the characteristic function.
\end{proposition}

\begin{theorem}
    The diameter of any component $G_{(g_p, g_q)}$ of $\Gamma_0$ is $2$.
\end{theorem}

\begin{theorem}
    The diameter of subgraph $\Gamma_1$ of the Bi-Cayley graph $\Gamma$ is $4$.
\end{theorem}

After obtaining the results above, we now examine the distance properties between two vertices within each of the subgraphs $\Gamma_0$ and $\Gamma_1$, in order to determine the diameter of the Bi-Cayley graph $\Gamma$ in general.

Let $g=(g_p, g_q)$, $h= (h_p, h_q) \in G$. We first investigate the distance between the two vertices in $\Gamma_1$. We have that
    \begin{align*}
            d_{\Gamma_1} ((g_p.g_q), (h_p, h_q)) = \begin{cases}
                1, & g_p = h_p \text{ or } g_q = h_q\\
                2, &g_p \not \equiv_p h_p \text{ and } g_q \not \equiv_q h_q \\
                3, & g_p \not \equiv_p h_p  \text{ and } g_q \equiv_q h_q  \hspace{0.3cm} \text{(vice versa)} \\
                4, & g_p \equiv_p h_p  \text{ and } g_q \equiv_q h_q.
            \end{cases}
        \end{align*}

Now, in $\Gamma_0$, if the two vertices lie in the same component, the distance is at most $2$. In the case that they lie in a different component, we have that
    \begin{enumerate}
        \item [$(i)$] the distance between the two is $5$ if one of these holds;
            \begin{itemize}
                \item [$(a)$] $g_p \not\equiv_p h_p$ and $g_q \equiv_q h_q, g_q \not\equiv_{q^2} h_q$,
                \item [$(b)$] $g_p \equiv_p h_p, g_p \not\equiv_{p^2} h_p$ and $g_q \not\equiv_q h_q$,
                \item [$(c)$] $g_p \not\equiv_p h_p$ and $g_q \equiv_q h_q, g_q \not\equiv_{q^2} h_q$,
                \item [$(d)$] $g_p \equiv_p h_p, g_p \not\equiv_{p^2} h_p$ and $g_q \not\equiv_q h_q$.
            \end{itemize}
        \item [$(ii)$] the distance between the two is at most $4$ if $(i)$ does not hold, i.e., $g_p \not\equiv_p h_p$ and $g_q \not\equiv_q h_q$.
    \end{enumerate}

The other possible case is that those two vertices lie in a different subgraph, then the maximum distance is $1 + \mathrm{diam}(\Gamma_1)$, which is equal to $5$. Therefore, these observation is formally stated in the following theorem. 

\begin{theorem}
    The diameter of $\Gamma$ is $5$.
\end{theorem}

We investigate properties related to the independent set and the independence number of the Bi-Cayley graph $\Gamma$. The independent sets of $\Gamma$ will first be examined on each of its subgraphs, namely $\Gamma_0$ and $\Gamma_1$. After determining the independent set within each subgraph, we will then analyze how these results relate to the independent sets of the entire Bi-Cayley graph $\Gamma$.

\begin{theorem}
    The independence number of the subgraph $\Gamma_0$ is $pq \min \{ p,q\}$. 
\end{theorem}

\begin{proof}
    We have already show that $\Gamma_0$ consists of $pq$ components, each of which are isomorphic to the cartesian product graph $K_p \square K_q$. It is clear that $\alpha (K_p \square K_q) = \min \{ p,q\}$. Therefore, the largest independent set is $pq \min \{ p,q\}$.
\end{proof}

The next step is to examine the independent set and the independence number of the subgraph $\Gamma_1$. Recall that the subgraph $\Gamma_1$ satisfies the isomorphism $\Gamma_1 \cong A \square B$ where $A \cong K_{p,p,\dots, p}$ and $B \cong K_{q,q,\dots, q}$. The graphs $A$ and $B$ are multipartite graphs with $p$ and $q$ parts, respectively. 
    
    Consequently, the vertex sets of $A$ and $B$ can be expressed as 
    \begin{align*}
    	V(A_i) &= \{ (i, \alpha) : \alpha = 0, 1, 2, \dots, p-1 \} \text{ and }   V(B_j) = \{ (j, \beta) : \beta = 0, 1, 2, \dots, q-1 \}.
    \end{align*}
    where $A_i$ denote the $i$-th part of the $p$-partite graph $K_{p,p, \dots , p}$ and $B_j$ denote the $j$-th part of the $q$-partite graph $K_{q,q, \dots , q}$. Then 
    \begin{align*}
    	\bigcup_{i=1}^{p} V(A_i) = V(A) = V(K_{p,p,\dots , p}) \text{ and } \bigcup_{j=1}^{q} V(B_j) = V(B) = V(K_{q,q,\dots , q}). 
    \end{align*} 
    Thus, the vertex set of $\Gamma_1 = A \square B$ can be written as
    \begin{align*}
    	V(A \square B) = \{((i, \alpha) , (j, \beta)) : i, \alpha = 0, 1, \dots, p-1; j, \beta = 0,1, \dots, q-1 \}. 
    \end{align*}
    It can be observed that distinct vertices $((i, \alpha), (j, \beta))$ and $((i', \alpha'), (j', \beta'))$ are adjacent if and only if either $i = i'$ \& $j \neq j'$ or $i \neq i$ \& $j = j$.  %one of the following conditions holds. 
    %\vspace{-0.2cm}
    %\begin{enumerate}
    %	\item [$(i)$] $i = i'$ and $j \neq j'$, or 
    %	\item [$(ii)$] $i \neq i'$ and $j = j'$.
    %\end{enumerate}

    \begin{remark}
    	For the remainder of this subsection, the notation $A \square B$ and $K_{p,p,\dots,p} \square K_{q,q,\dots,q}$ are used interchangeably as they refer to the same graph structure. This means that each of the corresponding vertices may be used interchangeably suited for the analysis. The vertex $((i, \alpha),(j, \beta)) \in V(K_{p,p,\dots,p} \square K_{q,q,\dots,q})$ will simply be written as $(i, \alpha; j, \beta)$.
    \end{remark}

    As a way to explicitly show the graph isomorphism of $A\square B$ and $K_{p,p,\dots,p} \square K_{q,q,\dots,q}$, we may define the function $f$ as follows. Define $f : V(A \square B) \rightarrow V(K_{p,p,\dots , p} \square K_{q,q,\dots , q})$ by
    \begin{align*}
    	f(g_p, g_q) = (i, \alpha;  j, \beta)
    \end{align*}
    for all $(g_p, g_q) \in V(A \square B)$ with $g_p = i + p \alpha$ and $g_q = j + q \beta$ with $i, \alpha \in \{0,1,\dots, p-1 \}$ and $j, \beta \in \{0,1,\dots , q-1 \}$.

    The following proposition acts as a motivation in constructing the independent sets in $\Gamma_1$. 

    \begin{proposition}
        Let $i \in \{0,1, \dots, p-1 \}$ and $j \in {0,1, \dots, q-1}$. Then the set $C(i, j) \subseteq V(K_{p,p,\dots, p} \square K_{q,q, \dots, q})$ defined as
    	\begin{align*}
    		C(i,j) = \{(i, \alpha; j, \beta) : \alpha = 0,1, \dots , p-1; \beta = 0, 1, \dots , q - 1\}
    	\end{align*}
    	is an independent set of the graph  $K_{p,p, \dots, p} \square K_{q,q, \dots, q}$ of size $|C(i,j)| = pq$.
    \end{proposition}

    The proposition above furnishes an independent set of size $pq$ in the subgraph $\Gamma_1$. Moreover, independent sets in $\Gamma_1$ may be expanded by taking unions of several sets $C(i,j)$, suitable choices of such unions produce larger independent sets of $\Gamma_1$.

    \begin{proposition}
    	Let $i,i' \in \{0,1, \dots, p-1 \}$ and $j,j' \in \{0,1, \dots, q-1 \}$. Any vertex in $C(i,j)$ is not adjacent to every vertex in $C(i', j')$ if and only if $i \neq i'$ and $j \neq j'$. 
    \end{proposition}

    \begin{remark}\label{independent sets construct}
        The set $\bigcup_{i = 0}^{\mathrm{min}\{p,q\}} C(i,i)$
    	%\begin{align*}
    	%	\bigcup_{i = 0}^{\mathrm{min}\{p,q\}} C(i,i)
    	%\end{align*}
    	is an independent set of size $pq \min \{p,q \}$.
    \end{remark}

    Remark \ref{independent sets construct} serves as the motivation for constructing the largest independent set within the subgraph $\Gamma_1$. This reasoning is based on the fact that every independent set in $\Gamma_1$ can be formed from a subset or the entirety of the sets $C(i,j)$.

    \begin{proposition}
        Let $S \subseteq V(K_{p,p,\dots, p} \square K_{q,q,\dots, q})$ be an independent set. Let $i \in \{0,1,\dots, p-1\}$, $j \in \{0,1,\dots, q-1\}$. Define $S_{i,j} = S \cap C(i,j)$, then the set of indices $\mathcal{I} = \{(i,j) : S_{i,j} \neq \emptyset \}$
    	%\vspace{-0.4cm}
    	%\begin{align*}
    	%	\mathcal{I} = \{(i,j) : S_{i,j} \neq \emptyset \}
    	%\end{align*}
    	does not contain indices that share the same first or second coordinate. 
    \end{proposition}

    Next, several conditions under which a subset $S \subseteq V(K_{p,p,\dots, p} \square K_{q,q,\dots,q})$ forms an independent set will be presented. Consider the following proposition. 
    
    \begin{proposition}\label{karakterisasi independensi}
    	For each  $(i, \alpha) \in V(K_{p,p,\dots, p})$, choose a subset $T_{(i, \alpha)} \subseteq V(K_{q,q,\dots, q})$ such that there exists unique $j' \in \{0,1,\dots, q-1\}$ satisfying $T_{(i, \alpha)} \subseteq B_{j'}$. An independent set $S \subseteq V(K_{p,p,\dots, p} \square K_{q,q,\dots, q})$ can be expressed as
    	\begin{align*}
    		S = \{(i, \alpha ; j, \beta) : (j, \beta) \in T_{(i, \alpha)} \}.
    	\end{align*}
    	The set $S$ is an independent set if and only if for every $(j, \beta) \in V(K_{q,q,\dots, q})$, there exists a unique $i' \in \{0,1,\dots, p-1 \}$ such that $\{(i, \alpha) \in V(K_{p,p,\dots,p}) : (j, \beta ) \in T_{(i, \alpha)}\} \subseteq A_{i'}$.
    \end{proposition}

    \begin{proof}
        \begin{enumerate}
    		\item [$(\Rightarrow)$] Suppose that $S$ is an independent set. Let $(j, \beta) \in V(K_{q,q,\dots,q})$. For the sake of contradiction, suppose there exist $(i, \alpha), (i', \alpha') \in V(K_{p,p,\dots, p})$ with $i \neq i'$ or they lie on a different part $A_i$ such that $(j, \beta) \in T_{(i, \alpha)}$ and $(j, \beta) \in T_{(i', \alpha')}$. This implies that $(i, \alpha; j, \beta), (i', \alpha' ; j, \beta) \in S$. This contradicts the independency property of $S$ since those two vertices are adjacent. Thus, it must be the case that for every $(i, \alpha)$ with $(j, \beta) \in T_{(i, \alpha)}$ lie in a same $A_i$ part. In other words, for every $(j, \beta) \in V(K_q,q,\dots,q)$, there exists unique $i' \in \{0,1,\dots, p-1 \}$ such that $\{(i, \alpha) \in V(K_{p,p,\dots,p}) : (j, \beta ) \in T_{(i, \alpha)}\} \subseteq A_{i'}$.
    		
    		\item [$(\Leftarrow)$] Suppose that there exists $j' \in \{0,1,\dots, q-1 \}$ and $i' \in \{0,1,\dots, p-1 \}$ such that for every $(i, \alpha) \in V(K_{p,p,\dots,p})$ and $(j, \beta) \in V(K_{q,q,\dots,q})$ implies $T_{(i, \alpha)} \subseteq B_{j'}$ and $\{(i, \alpha) \in V(K_{p,p,\dots,p}) : (j, \beta) \in T_{(i, \alpha)} \} \subseteq A_{i'}$. Let $(i, \alpha; j, \beta), (i', \alpha' ; j', \beta') \in S$, consider all of the possible cases. 
    		\begin{enumerate}
    			\item [$(i)$] If $(i, \alpha) = (i', \alpha')$ then $(j, \beta), (j', \beta') \in T_{(i, \alpha)} \subseteq B_{j'}$. Thus, $(j, \beta)$ and $(j', \beta')$ lie on the same $B_{j'}$ part. Therefore, $j = j'$ which means $(j, \beta)$ and $(j', \beta')$ are not adjacent in $K_{q,q,\dots,q}$. Hence, $(i, \alpha; j, \beta) $ and $ (i', \alpha' ; j', \beta')$ are not adjacent.
    			\item [$(ii)$] If $(j, \beta) = (j', \beta')$ then $(i, \alpha), (i', \alpha') \in \{(i, \alpha) \in V(K_{p,p,\dots,p}) : (j, \beta) \in T_{(i, \alpha)} \} \subseteq A_{i'}$. Thus, $(i, \alpha)$ and $(i', \alpha')$ lie on the same $A_i$ part. Therefore, $i=i'$, which implies $(i, \alpha)$ and $(i', \alpha')$ are not adjacent in $K_{p,p,\dots,p}$. Thus, $(i, \alpha; j, \beta) $ and $ (i', \alpha' ; j', \beta')$ are not adjacent.
    			\item [$(iii)$] If $(i, \alpha) \neq (i', \alpha')$ and $(j, \beta) \neq (j', \beta')$ then, clearly, $(i, \alpha; j, \beta) $ and $ (i', \alpha' ; j', \beta')$ are not adjacent.
    		\end{enumerate}
    		Thus, $S$ is an independent set. 
    	\end{enumerate}
    \end{proof}

    Analogous to the approach presented in Proposition \ref{karakterisasi independensi}, it can be observed that an independent set may be expressed as $S = \{ (i, \alpha; j, \beta) : (i, \alpha) \in R_{(j, \beta)} \}$
    %\begin{align*}
    %	S = \{ (i, \alpha; j, \beta) : (i, \alpha) \in R_{(j, \beta)} \}
    %\end{align*}
    where $R_{(j, \beta)} \subseteq V(K_{p,p,\dots, p})$ such that there exists a unique $i' \in \{0,1,\dots, p-1 \}$ satisfying $R_{(j, \beta)} \subseteq A_{i'}$. This observation can further be utilized to determine the maximum possible size of any independent set $S$ in the graph $K_{p,p,\dots, p} \square K_{q,q,\dots,q}$. The result is presented in the following proposition. 
    
    \begin{proposition}\label{batas atas independen set}
    	Consider the graph $K_{p,p,\dots, p} \square K_{q,q,\dots,q}$. For any arbitrary independent set $S$ in $K_{p,p,\dots, p} \square K_{q,q,\dots,q}$, we have $|S| \leq pq \min \{p,q \}.$
    	%\begin{align*}
    	%	|S| \leq pq \min \{p,q \}.
    	%\end{align*}
    \end{proposition}

    \begin{proof}
        In accordance to Proposition \ref{karakterisasi independensi}, any independent set $S \subseteq V(K_{p,p,\dots, p} \square K_{q,q,\dots, q})$ may be expressed as one of the following
    	\begin{align*}
    		S = \{(i, \alpha ; j, \beta) : (j, \beta) \in T_{(i, \alpha)} \} \text{ or } S = \{ (i, \alpha; j, \beta) : (i, \alpha) \in R_{(j, \beta)} \}.
    	\end{align*}
    	We shall first examine the case of an independent set of the first form. For all $(i, \alpha) \in A_i$, we have $T_{(i, \alpha)} \subseteq B_{j'}$ for unique $j' \in \{0,1,\dots, q-1 \}$, which implies $|T_{(i, \alpha)}| \leq q$
    	%\begin{align*}
    	%	|T_{(i, \alpha)}| \leq q
    	%\end{align*} 
    	for all $(i, \alpha) \in A_i$. Now, $\sum_{(i, \alpha) \in A_i} |T_{(i, \alpha)}| \leq pq.$
    	%\begin{align*}
    	%	\sum_{(i, \alpha) \in A_i} |T_{(i, \alpha)}| \leq pq.
    	%\end{align*}
    	Therefore,  
    	\begin{align*}
    		|S| = \sum_{i=1}^{p} \sum_{(i, \alpha) \in A_i} |T_{(i, \alpha)}| \leq \sum_{i=1}^{p} pq = p^2q.
    	\end{align*}
    	On the other hand, in the case that the independent set is of the second form, we analoguously have 
    	\begin{align*}
    		|S| = \sum_{j=1}^{q} \sum_{(j, \beta) \in B_j} |R_{(j, \beta)}| \leq \sum_{j=1}^{q} pq = pq^2.
    	\end{align*}
    	Hence, $|S| \leq p^2q$ and $|S| \leq pq^2$ simultanuously which implies that $|S| \leq pq \mathrm{min}\{p,q\}$.
    \end{proof}

    From Remark \ref{independent sets construct} and the upper bound for the independent set size, we conclude the following.

    \begin{theorem}
        The independence number of the subgraph $\Gamma_1$ is $pq \min \{p,q \}$.
    \end{theorem}

    The independence numbers of the subgraphs $\Gamma_0$ and $\Gamma_1$ of the Bi-Cayley graph $\Gamma$ have been determined. However, to establish the overall independence number of the Bi-Cayley graph $\Gamma$, it is necessary to examine how the independent sets of these subgraphs are constructed. For the independent set of $\Gamma_0$, the construction is relatively straightforward, one can simply select $\mathrm{min}\{p,q\}$ pairwise non-adjacent vertices from each component. In contrast, the construction of an independent set within the subgraph $\Gamma_1$ is more intricate. To address this, the following proposition is presented to characterize the membership properties of a maximum independent set in the subgraph $\Gamma_1$. 
    
    \begin{proposition}\label{bentuk akhir independen max}
    	If $S$ is the largest independent set in $K_{p,p,\dots,p} \square K_{q,q,\dots,q}$ then $S$ is the union of $\mathrm{min}\{p,q\}$ sets of $C(i,j)$ in which all of the indices $i$ and $j$ are distinct. 
    \end{proposition}
    
    \begin{proof}
    	Without the loss of generality, suppose that $S$ takes the form $S = \{(i, \alpha ; j, \beta) : (j, \beta) \in T_{(i, \alpha)} \}.$
    	%\begin{align*}
    	%	S = \{(i, \alpha ; j, \beta) : (j, \beta) \in T_{(i, \alpha)} \}.
    	%\end{align*}
    	By the definition, for each $T_{(i, \alpha)}$ there exists a unique $j' \in \{0,1,\dots,q-1\}$ such that $T_{(i,\alpha)} \subseteq B_{j'}$. By the independency property, there exists unique $i\ \in \{0,1,\dots,p-1\}$ such that the set $\{(i, \alpha) \in V(K_{p,p,\dots,p}) : (j, \beta) \in T_{(i, \alpha)} \}$ lies in the single part $A_{i'}$. To clarify the notation, those unique $i'$ and $j'$ will be denoted by $\theta(j)$ and $\gamma(i)$, respectively. This means that for all $(i, \alpha) \in V(K_{p,p, \dots, p})$ and $(j, \beta) \in V(K_{q,q,\dots, q})$, we have $T_{(i, \alpha)} \subseteq B_{\gamma(i)}$ and $\{(i, \alpha) \in V(K_{p,p,\dots,p}) : (j, \beta) \in T_{(i, \alpha)} \} \subseteq A_{\theta(j)}$. 
    	
    	Next, considering the case $p < q$, 
    	\begin{align*}
    		|S| &= \sum_{k=1}^{p} \sum_{(i, \alpha) \in A_i} |T_{(i, \alpha)}| = p^2q 
    		\iff |T_{(i, \alpha)}| = q \hspace{0.2cm}\text{for all } (i, \alpha) \in A_i.  
    	\end{align*}
    	This implies $T_{(i, \alpha)} = B_{\gamma(i)}$, or equivalently, $T_{(i, \alpha)}$ is the entire part $B_{\gamma(i)}$ of the graph $K_{q,q,\dots,q}$.
    	
    	For an index $j \in \{0,1,\dots,q\}$, define the set $U_{j} = \{(i, \alpha) \in V(K_{p,p,\dots,p}) : \gamma(i) = j \},$
    	%\begin{align*}
    	%	U_{j} = \{(i, \alpha) \in V(K_{p,p,\dots,p}) : \gamma(i) = j \},
    	%\end{align*}
    	that is, the set in which $(i, \alpha)$ corresponds to the whole part $B_{j}$. Hence, the set $\{ (i, \alpha) : (j, \beta) \in T_{(i, \alpha)} \}$ is $U_j$. From the independency property, $U_j$ is contained in part $A_{\theta(j)}$. Therefore, every $(i, \alpha)$ that corresponds to the index $\gamma(i)$ is contained in single part $A_{\gamma(i)}$. 
    	
    	Next, we claim that each index $j$ only corresponds to a single part $A_{\theta(j)}$. For the sake of contradiction, suppose that there exist a vertex in each $A_{i}$ and $A_{i'}$ with $i \neq i'$ that correspond to the same index $j$, that is $\gamma(i)=\gamma(i')$. Take the vertex $(i, \alpha) \in A_i \cap U_j$ and vertex $(i', \alpha') \in A_{i'} \cap U_j$. Observe that, for all $(j, \beta) \in B_{\gamma(i)}$, the vertices $(i, \alpha; j, \beta), (i', \alpha'; j, \beta) \in S$ are adjacent. This contradicts the fact that $S$ is an independent set. Therefore, it must be the case that $\gamma(i) \neq \gamma(i')$ for all $i \neq i'$. In other words, the index $j$ only corresponds to a single part $A_{\theta(j)}$. Equivalently, the function $\varphi : \mathcal{A} \rightarrow \{0,1,\dots q-1\}$ such that 
    	\begin{align*}
    		\varphi (A_i) = \gamma(i)
    	\end{align*}
    	where $\mathcal{A} = \{A_i : i = 0,1, \dots, p-1\}$ is injective.
    	
    	To achieve maximality, each of $(i, \alpha) \in A_i$ corresponds to the entire part $B_{\varphi(A_i)}$. Thus, 
    	\begin{align*}
    		\bigcup_{(i, \alpha) \in A_i} \left( \{(i, \alpha)) \} \times T_{(i, \alpha)} \right) = C(i, \varphi(A_i)) \subseteq S.
    	\end{align*}
    	Therefore, the largest independent set in the case of $p<q$ is $\bigcup_{i=1}^{p} C(i, \varphi(A_i)).$
    	%\begin{align*}
    	%	\bigcup_{i=1}^{p} C(i, \varphi(A_i)).
    	%\end{align*}
    \end{proof}

    \begin{corollary}
        Let $S \subseteq V(K_{p,p, \dots, p} \square K_{q,q,\dots, q})$ be any maximal independent set. Then 
    	\begin{align*}
    		S = \bigcup_{(i,j) \in \mathcal{I}} C(i,j).
    	\end{align*}
    \end{corollary}

    After obtaining the constructions of the maximal independent sets for the subgraphs $\Gamma_0$ and $\Gamma_1$, the independence number of the entire Bi-Cayley graph $\Gamma$ can now be determined. Before doing so, however, the following observation should be noted. The independent set in the subgraph $\Gamma_1$ is formed from the sets $C(i,j)$. Under the isomorphism $\Gamma_1 \cong K_{p,p,\dots,p} \square K_{q,q,\dots,q}$, each set $C(i,j)$ corresponds to the following subset.
    \begin{align*}
    	C(i,j) \mapsto \{(i+ p\alpha, j+\beta): \alpha = 0,1,\dots,p-1; \beta = 0,1,\dots,q-1 \} \subseteq \mathbb{Z}_{p^2} \times \mathbb{Z}_{q^2}.
    \end{align*}
    Upon closer examination, it can be observed that the aforementioned set is, in fact, the component of the subgraph $\Gamma_0$ indexed by the vertex $(i,j)$. Considering $(i,j) \in \mathbb{Z}_{p^2} \times \mathrm{Z}_{q^2}$, we have $V(G_{(i,j)}) = \{(i+ p\alpha, j+\beta): \alpha = 0,1,\dots,p-1; \beta = 0,1,\dots,q-1 \}$,
    %\begin{align*}
    %	V(G_{(i,j)}) = \{(i+ p\alpha, j+\beta): \alpha = 0,1,\dots,p-1; \beta = 0,1,\dots,q-1 \}
    %\end{align*}
    where $G_{(i,j)}$ denotes a component of the subgraph $\Gamma_0$, previously referred to as $G_{(g_p,g_q)}$. Hence, it follows that the independent sets in each subgraph, $\Gamma_0$ and $\Gamma_1$, are determined by vertex sets of a similar structural form, differing only in the manner in which their respective components are connected. This structural correspondence leads to the following result. 
    
    \begin{theorem}\label{bilangan independensi graf bicayley}
    	The independence number of $\Gamma$ is $2pq \min \{p,q \} - \mathrm{min}\{p^2,q^2\}$.
    \end{theorem}

    In simpler terms, there are overlapping vertices in each construction of the largest independent sets of each subgraph. Thus, we cannot take full sized, i.e., $pq \min \{p,q\}$ independent set of each subgraphs. Those overlapping vertices constitute to $(\min \{p, q\})^2$ of the total construction in each. 
    
\subsection{Extension to Arbitrary Finite Group}
In the preceding sections, we have focused on the study of Bi-Cayley graphs over cyclic groups, where several structural and combinatorial properties were derived explicitly. These results provide a concrete foundation for understanding the behavior of Bi-Cayley graphs in a well-structured algebraic setting. We now turn to a more general framework by considering Bi-Cayley graphs over arbitrary finite groups. Having established the results for cyclic groups, we wish to extend these results to Bi-Cayley graphs over arbitrary finite groups. The general results to be presented include fundamental counting properties, connectivity, the clique number, as well as an overview of results concerning the chromatic number. 

\begin{remark}
	It is to be noted that the Bi-Cayley graphs considered in this section are those for which the set $S_3$ contains only the identity element of the finite group, unless stated otherwise.
\end{remark}
%
%In this subsection, we examine the fundamental counting properties and the general structural features of Bi-Cayley graphs over finite groups. These basic characteristics describe essential aspects of the graph, including the number of vertices and edges as well as the degree distribution, and provide a clear understanding of the overall organization of Bi-Cayley graphs. The results presented here apply to Bi-Cayley graphs over arbitrary finite groups and form a foundational component of their combinatorial structure.
%
\begin{proposition}
    Let $G$ be a finite group and let $S_1, S_2 \subseteq G$ such that $e_G \notin S_1, S_2$ and $S_1^{-1} = S_1$, $S_2^{-1} = S_2$, and $S_3 = \{e_G \}$. Let $\Gamma = \mathrm{Cay}(G, S_1, S_2, S_3)$, then $|V(\Gamma)| = 2|G|$ and $|E(G)| = \dfrac{|G|}{2}(|S_1|+|S_2|+2).$
		%\begin{align*}
		%	|V(\Gamma)| = 2|G| \text{ and } |E(G)| = \dfrac{|G|}{2}(|S_1|+|S_2|+2).
		%\end{align*}
\end{proposition}

The preceding proposition can be trivially concluded through basic set theory and \textit{Handshaking Lemma}.

\begin{remark}
	Utilizing previous notation, we still adopt the notation from the cyclic group cases that $\Gamma_0$ and $\Gamma_1$ respectively denote Cayley graphs $\mathrm{Cay}(G, S_1)$ and $\mathrm{Cay}(G, S_2)$. Moreover, $\Gamma$ denotes the Bi-Cayley graph $\mathrm{Cay}(G; S_1, S_2, S_3)$.
\end{remark}

\begin{proposition}
    Let $G$ be a finite group and let $\Gamma = \mathrm{Cay}(G, S_1, S_2, S_3)$ be its associated Bi-Cayley graph. Then
		\begin{align*}
			deg_{\Gamma}((i,g)) = \begin{cases}
				|S_1| + 1 & \text{ if } i = 0 \\
				|S_2| + 1 & \text{ if } i = 1 
			\end{cases}
		\end{align*}
	for all $(i, g) \in V(\Gamma)$.
\end{proposition}

\begin{theorem}
    Let $G$ be a finite group and let $\Gamma = \mathrm{Cay}(G, S_1, S_2, S_3)$ be its associated Bi-Cayley graph. Then the following assertions holds
		\begin{enumerate}
			\item [$(i)$] $\Gamma$ is regular if and only if $|S_1|= |S_2|$.
			\item [$(ii)$] $\Gamma$ is biregular if and only if $|S_1| \neq |S_2|$.
		\end{enumerate}    
\end{theorem}

%\subsection{Graph Connectivity}
%With the fundamental counting properties and general structural features of Bi-Cayley graphs now in place, we turn to an examination of their global structural behavior. One of the most basic yet significant global properties of a graph is its connectivity, which determines whether the graph consists of a single connected component or decomposes into several disjoint parts. In the context of Bi-Cayley graphs, connectivity is closely tied to the algebraic structure of the underlying group and the subsets that define the graph. We therefore proceed to investigate the conditions under which a Bi-Cayley graph is connected.
%
We show that the connectivity of the entire Bi-Cayley graph depends solely on the connectivity of one of its Cayley subgraphs. Consequently, it suffices to examine the connectivity of each subgraph individually. This result is stated formally in the following theorem.

\begin{theorem}
	Let $G$ be a finite group. The associated Bi-Cayley graph $\Gamma$ is connected if and only if either $\Gamma_0$ or $\Gamma_1$ is connected.
\end{theorem}

It is rather intuitive result since we can trivially construct a path if either of the subgraphs is connected simply by jumping to the connected one, then back to the other side, if needed. Conversely, if both of $\Gamma_0$ and $\Gamma_1$ are disconnected, then there would be no configuration in which the entire graph is connected. In terms of its pure algebraic properties, the following remark is an immediate corollary to the theorem above. 

\begin{remark}
    Let $G$ be a finite group and let $S_1, S_2 \subseteq G$ such that $e_G \notin S_1, S_2$ and $S_1^{-1} = S_1$, $S_2^{-1} = S_2$, and $S_3 = \{e_G \}$. The Bi-Cayley graph $\Gamma$ is connected if and only if either $S_1$ or $S_2$ generates $G$.
\end{remark}

Closely related to graph connectivity, the Eulerian property of a Bi-Cayley graph over an arbitrary finite group can be characterized as follows.

\begin{theorem}\label{eulerianity of bicayley}
	Let $G$ be a finite group and let $\Gamma$ be the associated Bi-Cayley graph. Then $\Gamma$ is an Eulerian graph if and only if $S_1$ and $S_2$ both contain odd number of elements.
\end{theorem}

%\subsection{Clique Number}
%The clique number is a central invariant in graph theory, as it measures the extent of local completeness and plays an important role in both extremal and structural problems. For Bi-Cayley graphs over finite groups, the clique number reflects how the algebraic structure of the group and the choice of subsets determine the possible configurations of fully adjacent vertex sets.

\begin{proposition}\label{clique size >2}
	Let $G$ be a finite group and let $\Gamma$ be the associated Bi-Cayley graph. The clique of size greater than $2$, if there exists, must lie entirely within one side.
\end{proposition}
The approach is very much similar to the cyclic group case, which is still applicable here.

Following the proposition, we can now determine the clique number of the entire Bi-Cayley structure. This is formally stated in the preceeding theorem. 

\begin{theorem}
	Let $G$ be a finite group and let $\Gamma$ be the associated Bi-Cayley graph. Then $\omega(\Gamma) = \mathrm{max}\{\omega(\Gamma_0), \omega(\Gamma_1), 2\}.$
%		\begin{align*}
%			\omega(\Gamma) = \mathrm{max}\{\omega(\Gamma_0), \omega(\Gamma_1), 2\}.
%		\end{align*}
\end{theorem}

\begin{proof}
    First, we seek to prove that $\omega(\Gamma) \geq \mathrm{max}\{\omega(\Gamma_0), \omega(\Gamma_1), 2\}$. We prove each part separately. Since $\Gamma_0$ is an induced subgraph of $\Gamma$, any clique in $\Gamma_0$ remains a clique in $\Gamma$. Thus, $\omega(\Gamma) \geq \omega(\Gamma_0)$. Similar argument gives $\omega(\Gamma) \geq \omega (\Gamma_1)$. Furthermore, for any $g \in G$, $(0,g)$ and $(1,g)$ forms an edge which is a clique of size $2$. Hence $\omega(\Gamma) \geq 2$. Therefore, the desired $\omega(\Gamma) \geq \mathrm{max}\{\omega(\Gamma_0), \omega(\Gamma_1), 2\}$ is obtained.

    Next, we wish to show that $\omega(\Gamma) \leq \mathrm{max}\{\omega(\Gamma_0), \omega(\Gamma_1), 2\}$. Let $K$ be a clique in $\Gamma$. Consider all possibilities for how $K$ meets the two sides. If $K \subseteq V(\Gamma_0)$, then $K$ is a clique in $\Gamma_0$, so $|K| \leq \omega(\Gamma_0)$. Similarly, if $K \subseteq V(\Gamma_1)$, then $K$ is a clique in $\Gamma_1$, so $|K| \leq \omega(\Gamma_1)$. Lastly, if $K \cap (V(\Gamma_0) \cup V(\Gamma_1)) \neq \emptyset$, then by Proposition \ref{clique size >2}, it follows that $|K| \leq 2$. In all cases, we have $|K| \leq \mathrm{max}\{\omega(\Gamma_0), \omega(\Gamma_1), 2\}.$ Since $K$ is arbitrary, it follows that $\omega(\Gamma) \leq \mathrm{max}\{\omega(\Gamma_0), \omega(\Gamma_1), 2\}.$
	 Therefore, it can be concluded that $\omega(\Gamma) = \mathrm{max}\{\omega(\Gamma_0), \omega(\Gamma_1), 2\}.$
	 	%\begin{align*}
	 	%	\omega(\Gamma) = \mathrm{max}\{\omega(\Gamma_0), \omega(\Gamma_1), 2\}.
	 	%\end{align*}
\end{proof}

%\subsection{Chromatic Number}
We have analyzed the chromatic number of Bi-Cayley graphs arising from certain cyclic groups, where explicit colorings could be constructed. We now extend this investigation to Bi-Cayley graphs over arbitrary finite groups. %The chromatic number is a fundamental graph invariant that measures the minimum number of colors required for a proper vertex coloring and reflects both local adjacency constraints and global structural features of the graph. In the generalized finite group setting, the chromatic number is influenced by the interplay between the algebraic structure of the group and the defining subsets of the Bi-Cayley graph.

\begin{theorem}
	Let $G$ be a finite group and let $\Gamma$ be the associated Bi-Cayley graph. Then $\mathrm{max}\{\chi(\Gamma_0), \chi(\Gamma_1)\} \leq  \chi(\Gamma) \leq 1 + \mathrm{max}\{\chi(\Gamma_0), \chi(\Gamma_1)\}.$
		%\begin{align*}
		%	\mathrm{max}\{\chi(\Gamma_0), \chi(\Gamma_1)\} \leq  \chi(\Gamma) \leq 1 + \mathrm{max}\{\chi(\Gamma_0), \chi(\Gamma_1)\}. 
		%\end{align*}
\end{theorem}

\begin{proof}
    We will establish the lower bound for the chromatic number of $\Gamma$. Since both $\Gamma_0$ and $\Gamma_1$ are induced subgraphs of $\Gamma$, any proper coloring of $\Gamma$, when restricted to either side, gives a proper coloring of that induced subgraph. Hence, $\chi(\Gamma) \geq \chi(\Gamma_0)$ and $\chi(\Gamma) \geq \chi(\Gamma_1).$ Thus, taking the maximum yields the desired lower bound.

    Next, we will establish the upper bound for the chromatic number of $\Gamma$. Let $k = \mathrm{max}\{\chi(\Gamma_0), \chi(\Gamma_1)\}$. Let $c_0 : V(\Gamma_0) \rightarrow \{1, 2, \dots, k\}$ and $c_1 : V(\Gamma_1) \rightarrow \{1, 2, \dots, k\}$
			%\begin{align*}
			%	c_0 : V(\Gamma_0) \rightarrow \{1, 2, \dots, k\} \text{ and } c_1 : V(\Gamma_1) \rightarrow \{1, 2, \dots, k\}
			%\end{align*}
		be a proper coloring of $\Gamma_0$ and $\Gamma_1$, respectively. Recall the structure of cross edges, that is, $(0,g)$ is adjacent to $(1,g)$ for all $g \in G$. Thus a conflict of color occurs only if $c_0((0,g)) = c_1((1,g))$ 
			%\begin{align*}
			%	c_0((0,g)) = c_1((1,g))
			%\end{align*}
		for some $g \in G$. Define a coloring $c : V(\Gamma) \rightarrow \{1,2,\dots, k+1\}$ 
			%\begin{align*}
			%	c : V(\Gamma) \rightarrow \{1,2,\dots, k+1\}
			%\end{align*}
		as follows, 
			\begin{align*}
				c((i,g)) = \begin{cases}
					c_0((0,g)) & \text{ if } i = 0 \\
					c_1((1,g)) & \text{ if } i=1 \text{ and } c_1((1,g)) \neq c_0((0,g)) \\
					k+1 & \text{ if } i=1 \text{ and } c_1((1,g)) = c_0((0,g)).
				\end{cases}
			\end{align*}
		Since each side has a chromatic number less than or equal to $k$, $c$ is a proper coloring of the Bi-Cayley graph $\Gamma$. Thus, we have that $\chi(\Gamma) \leq 1 + \mathrm{max}\{\chi(\Gamma_0), \chi(\Gamma_1)\}.$ 
			%\begin{align*}
			%	\chi(\Gamma) \leq 1 + \mathrm{max}\{\chi(\Gamma_0), \chi(\Gamma_1)\}.
			%\end{align*}
	Therefore, we have $\mathrm{max}\{\chi(\Gamma_0), \chi(\Gamma_1)\} \leq  \chi(\Gamma) \leq 1 + \mathrm{max}\{\chi(\Gamma_0), \chi(\Gamma_1)\}.$
		%\begin{align*}
		%	\mathrm{max}\{\chi(\Gamma_0), \chi(\Gamma_1)\} \leq  \chi(\Gamma) \leq 1 + \mathrm{max}\{\chi(\Gamma_0), \chi(\Gamma_1)\}.
		%\end{align*}
\end{proof}

%\section{Special Cases of $S_3$ being Involutions}
We now extend our consideration to a broader and more structurally interesting setting. In contrast to the previously studied case, where the third connecting set was restricted to contain only the identity element, we shall now allow it to encode richer internal symmetries of the group. In particular, we focus on the situation in which the connecting set $S_3$ is chosen to be the set of all involutions of the finite group $G$. This choice introduces additional interactions between the two components of the Bi-Cayley structure and, as will be seen, leads to several significant modifications in the resulting graph properties.

Recall that an element $a \in G$ is called an involution if it has order two, equivalently, $a$ is an element satisfying $a^2 = e_G$ and $a \neq e_G$, so that $a$ is its own inverse. The set of all involutions of $G$ is denoted by $I(G)$. Naturally, the set $I(G)$ is inverse closed. In this setting, we deliberately chose $S_3 = I(G)$. Several characteristics of the associated Bi-Cayley graph $\mathrm{Cay} (G; S_1, S_2, I(G))$ will be presented, though it is limited spanning from general structure to clique number. 

%\subsection{General Structure}
Let $G$ be a finite group and $I(G)$ the set of its involutions. Let $S_1$ and $S_2$ be the standard connecting sets of the Bi-Cayley graph $\Gamma_{I(G)} = \mathrm{BiCay}(G; S_1, S_2, I(G))$. Now, we introduce a new subgraph structure other than the $\Gamma_0$ and $\Gamma_1$ subgraph. Since the chosen set $S_3$ is more intricate than what we are discussed before, we introduce subgraph $\Gamma^{ce}$, as we call it, the cross-edges subgraph. The cross-edges subgraph is defined to be a graph where we remove all edges in $\mathrm{Cay}(G,S_1)$ and $\mathrm{Cay}(G,S_2)$. %with $V(\Gamma^{ce}) = V(\Gamma)$ and  
   % \begin{align*}
   %     (i,x)(j,y) \in E(\Gamma^{ce}) \iff i \neq j \text{ and } xy^{-1} \in I(G) 
   % \end{align*}
%for all $(i,x), (j,y) \in V(\Gamma^{ce})$.

\begin{proposition}
    The cross-edges subgraph $\Gamma^{ce}$ is a $|I(G)|$-regular graph.
\end{proposition}

It is clear that for each $(0,g) \in V(\Gamma^{ce})$, it is adjacent to $(1, gs)$ for all $s \in I(G)$ hence the degree of $(0,g)$ is $|I(G)|$. The following theorem explain all the basic counting of the whole Bi-Cayley graph $\Gamma$. 

\begin{theorem}
    Let $G$ be a finite group and let $\Gamma_{I(G)} = \mathrm{Cay}(G; S_1, S_2, I(G))$ be the associated Bi-Cayley graph. Then $|V(\Gamma_{I(G)})| = 2|G|$ and $|E(\Gamma_{I(G)})| = \dfrac{|G|}{2}(|S_1|+|S_2|+ 2|I(G)|).$
        %\begin{align*}
         %   |V(\Gamma_{I(G)})| = 2|G| \text{ and } |E(\Gamma_{I(G)})| = \dfrac{|G|}{2}(|S_1|+|S_2|+ 2|I(G)|).
        %\end{align*}
\end{theorem}

It is also clear that the previous results on the degree and regularity also persists, hence the following also holds.

\begin{proposition}
    Let $G$ be a finite group and let $\Gamma_{I(G)} $ be its associated Bi-Cayley graph. Then
		\begin{align*}
			deg_{\Gamma_{I(G)}}((i,g)) = \begin{cases}
				|S_1| + |I(G)| & \text{ if } i = 0 \\
				|S_2| + |I(G)| & \text{ if } i = 1 
			\end{cases}
		\end{align*}
	for all $(i, g) \in V(\Gamma)$.
\end{proposition}

\begin{theorem}
    Let $G$ be a finite group and let $\Gamma_{I(G)} $ be its associated Bi-Cayley graph. Then the following assertions holds
		\begin{enumerate}
			\item [$(i)$] $\Gamma_{I(G)}$ is regular if and only if $|S_1|= |S_2|$.
			\item [$(ii)$] $\Gamma_{I(G)}$ is biregular if and only if $|S_1| \neq |S_2|$.
		\end{enumerate}    
\end{theorem}

%\subsection{Graph Connectivity}

\begin{proposition}\label{ce connected if I(G) generates}
    The cross-edges subgraph $\Gamma^{ce}$ is connected if and only if $I(G)$ generates the whole group. 
\end{proposition}

If a group is generated by its involutions, such as the Symmetric group $\mathrm{Sym}_3$, then for any choice of connecting sets $S_1$ and $S_2$, the associated Bi-Cayley graph $\Gamma$ will always be connected. This is an immediate result from Proposition \ref{ce connected if I(G) generates}. Moreover, it serves as a motivation to further furnish the characterization of the connectivity of the whole Bi-Cayley structure. 

\begin{theorem}
    Let $G$ be a finite group and let $\Gamma_{I(G)}$ be the associated Bi-Cayley graph. Set $H = \langle S_1 \cup S_2 \cup I(G) \rangle.$
        %\begin{align*}
         %   H = \langle S_1 \cup S_2 \cup I(G) \rangle.
        %\end{align*}
    The components of $\Gamma_{I(G)}$ are precisely $\{0,1 \} \times gH$ 
        %\begin{align*}
         %   \{0,1 \} \times gH
        %\end{align*}
    for all $g \in G$. In particular, $\Gamma_{I(G)}$ is connected if and only if $H$ generates $G$. 
\end{theorem}

%\subsection{Clique Number}
In general, when we talk about the lower bound for a clique in the Bi-Cayley graph, since the subgraph $\Gamma_0$ and $\Gamma_1$ are induced subgraph, we trivially conclude that $\omega(\Gamma_{I(G)}) \geq \max \{ \omega(\Gamma_0), \omega(\Gamma_1) \}.$
    %\begin{align*}
    %    \omega(\Gamma_{I(G)}) \geq \max \{ \omega(\Gamma_0), \omega(\Gamma_1) \}.
    %\end{align*}

In this setting, a clique of size greater than two may be constructed from both side since the edges connecting the two sides now need not necessarily be a perfect matching. Thus we will define a new terms regarding the clique that constitute to both of the sides.

\begin{definition}
    A clique $C \subseteq V(\Gamma_{I(G)})$ is called mixed clique if it contains vertices from both $0$ and $1$-side. 
\end{definition}

\begin{definition}
    A subset $S \subseteq G$ is involution-separating if $S \cap (H \setminus \{e_G \}) = \emptyset$
        %\begin{align*}
        %    S \cap (H \setminus \{e_G \}) = \emptyset
        %\end{align*}
    for every nontrivial elementary abelian 2-subgroup $H$ of $G$.
\end{definition}

Several claims concerning the characterization of mixed clique will be presented, though for now we may only give the results with some restriction. 

\begin{proposition}
    Let $G$ be a finite group and $\Gamma_{I(G)}$ be the associated Bi-Cayley graph such that $S_1$ and $S_2$ are involution-separating. Let $C$ be a mixed clique containing a vertex $(1,h)$ (resp. $(0,h)$). Suppose that, for some $k \in \mathbb{N}$, 
        \begin{align*}
            (0, g_i) = (0, a_ih) \in C, \hspace{0.3cm}i= 1,2, \dots , k \text{ (resp. }(1,g_i) = (1,ha_i) \in C)
        \end{align*}
    where each $a_i \in I(G)$. If there exists distinct indices $i \neq j$ such that the involutions $a_i$ and $a_j$ commute, then $|C \cap (\{0 \} \times G)| \leq 2 \text{ (resp. } |C \cap (\{1 \} \times G)| \leq 2).$
        %\begin{align*}
        %    |C \cap (\{0 \} \times G)| \leq 2 \text{ (resp. } |C \cap (\{1 \} \times G)| \leq 2).
        %\end{align*}
\end{proposition}

\begin{proof}
    We prove the bound for $\{0 \} \times G$, the argument for $\{ 1 \} \times G$ is similar. Suppose, for contradiction, that there exists a mixed clique $C$ such that $|C \cap (\{0 \} \times G)| \geq 3.$
        %\begin{align*}
        %    |C \cap (\{0 \} \times G)| \geq 3.
        %\end{align*}
    Then there exists three distinct elements $(0, g_1), (0, g_2), (0,g_3) \in C$ and at least one vertex $(1, h) \in C$. Since $C$ is a clique, then for each $i = 1,2, \dots, k$, $hg_i^{-1} \in I(G).$ 
        %\begin{align*}
        %    hg_i^{-1} \in I(G).
        %\end{align*}
    We denote $a_i := hg_i^{-1}$. Thus, each $a_i$ is an involution. Now, for any $i, j \in \{1,2, 3 \}$, $g_jg_i^{-1} = (hg_j^{-1})^{-1}(hg_i^{-1}) = a_j^{-1}a_i = a_ja_i.$
        %\begin{align*}
        %    g_jg_i^{-1} = (hg_j^{-1})^{-1}(hg_i^{-1}) = a_j^{-1}a_i = a_ja_i.
        %\end{align*}
    Therefore, for all distinct $i,j$, $a_ia_j \in S_0.$
        %\begin{align}
         %   a_ia_j \in S_0.
        %\end{align}

    Since there exist indices $i,j$ such that $a_i$ and $a_j$ commute, consider the subgroup $H = \langle a_i, a_j \rangle.$
        %\begin{align*}
        %    H = \langle a_i, a_j \rangle.
        %\end{align*}
    It is clear that $a_i a_j \in H \setminus \{e_G \}$ and in itself is a subgroup generated by commuting involutions, thus $H$ is elementary abelian $2$-subgroup. Therefore, $a_i a_j \in S_0 \cap (H \setminus \{e_G \})$
        %\begin{align*}
        %    a_i a_j \in S_0 \cap (H \setminus \{e_G \})
        %\end{align*}
    which contradicts the involution-separating property of $S_0$. Thus, the assumption is false and we conclude that $|C \cap (\{ 0\} \times G)| \leq 2.$
    %    \begin{align*}
    %        |C \cap (\{ 0\} \times G)| \leq 2.
    %    \end{align*}
\end{proof}

Though stronger results may be attained, the restriction of involution-separating property and commuting involution gives us a glimpse of the more general results on the clique number. Moreso, if the suitable condition that the restriction does occur, we have that 
    \begin{align*}
        \omega(\Gamma_{I(G)}) \leq \max \{ \omega(\Gamma_0), \omega(\Gamma_1), 4 \}.
    \end{align*}
where $4$ is the clique number of $\Gamma^{ce}$.
%
%We then proceed to present several example regarding the clique number. 
%
%\begin{example}
%    
%\end{example}

\section{Conclusion}\label{sec:conclusion}
This paper presents some results on Bi-Cayley graphs over cyclic groups of order $p^2q^2$ with connection sets defined by element orders. Explicit results are obtained for several structural and combinatorial parameters, revealing how order-based restrictions influence global graph properties. The analysis also suggests natural extensions to Bi-Cayley graphs over more general finite groups.
%\vspace{0.7cm}
%\\
%\textbf{Acknowledgements} The author would like to express sincere gratitude to . . . . .

\end{document}